\documentclass[12pt]{article}

\usepackage[english]{babel}

\usepackage[letterpaper,top=2cm,bottom=2cm,left=3cm,right=3cm,marginparwidth=1.75cm]{geometry}

\usepackage{amsmath, amssymb,amsthm}
\usepackage{graphicx}
\usepackage[colorlinks=true, allcolors=blue]{hyperref}
\usepackage{color}

\theoremstyle{definition}

\theoremstyle{plain}
\newtheorem{example}{Example}[section]
\newtheorem{theorem}{Theorem}[section]
\newtheorem{definition}[theorem]{Definition}
\newtheorem{proposition}[theorem]{Proposition}
\newtheorem{lemma}[theorem]{Lemma}
\newtheorem{corollary}[theorem]{Corollary}
\newtheorem{remark}{Remark}[section]
\newtheorem{assumption}{Assumption}[section]

\definecolor{darkyellow}{rgb}{0.7,0.7,0.1}

\definecolor{darkgreen}{rgb}{0.1,0.7,0.1}

\definecolor{darkred}{rgb}{0.7,0.1,0.1}

\definecolor{darkblue}{rgb}{0.1,0.1,0.7}

\DeclareMathOperator{\Hess}{Hess}

\DeclareMathOperator{\Var}{Var}
\DeclareMathOperator{\tv}{\textsc{tv}}
\DeclareMathOperator{\tmix}{{\rm t}_\textsc{mix}}
\DeclareMathOperator{\wmix}{{\rm w}_\textsc{mix}}

\DeclareMathOperator{\Ent}{Ent}
\DeclareMathOperator{\Varent}{Varent}

\newcommand{\R}{\mathbb{R}}
\newcommand{\N}{\mathbb{N}}

\newcommand{\E}{\mathbb{E}}
\renewcommand{\P}{\mathbb{P}}

\newcommand{\cA}{\mathcal{A}}
\newcommand{\cB}{\mathcal{B}}
\newcommand{\cP}{\mathcal{P}}
\newcommand{\cE}{\mathcal{E}}

\newcommand{\cT}{\mathcal{T}}

\newcommand{\TV}{\operatorname{TV}}
\newcommand{\Fix}{\operatorname{Fix}}
\newcommand{\He}{\operatorname{Hell}}

\DeclareMathOperator{\cL}{\mathcal{L}}

\newcommand{\Lip}{\operatorname{Lip}}

\title{Cutoff beyond non-negative curvature}
\author{Max Fathi\thanks{Université Paris Cité and Sorbonne Université, CNRS, Laboratoire Jacques-Louis Lions and Laboratoire de Probabilités, Statistique et Modélisation, F-75013 Paris, France and Institut Universitaire de France (IUF), \texttt{fathi@lpsm.paris}}, Arnaud Guillin\thanks{Laboratoire de Mathématiques Blaise Pascal (UMR 6620 CNRS), Université Clermont Auvergne, 3, Place Vasarely, Aubière, 63178, France, \texttt{arnaud.guillin@uca.fr}}, Cyril Labbé\thanks{Université Paris Cité and Sorbonne Université, CNRS, Laboratoire de Probabilités, Statistique et Modélisation, F-75013 Paris, France and Institut Universitaire de France (IUF), \texttt{clabbe@lpsm.paris}} and Justin Salez\thanks{CEREMADE, CNRS, Université Paris-Dauphine, PSL Research University, Place de Lattre de Tassigny, 75016 Paris, France and Institut Universitaire de France (IUF), \texttt{justin.salez@dauphine.psl.eu}}}

\begin{document}
\maketitle

\begin{abstract}
We study how a relaxed gradient estimate is enough to prove cutoff for several new families of Markov processes, within the varentropy framework of \cite{Sal25}. In particular, we show how certain regularization estimates from hypocoercivity theory allow one to prove cutoff for kinetic dynamics. Moreover, we prove cutoff for some families of diffusions with \emph{non-convex} potentials and for mean-field dynamics, as well as for biased random walks on the symmetric group, giving a simple proof of a far-reaching extension of a result of \cite{NeYa24}. 
\end{abstract}

\setcounter{tocdepth}{2}
\tableofcontents

\section{Introduction}

The cutoff phenomenon is a remarkable abrupt transition from out-of-equilibrium to equilibrium displayed by certain families of Markov processes.  Discovered in the early eighties in the context of card shuffling~\cite{Ald83,AldDia86}, it has since then been established in a variety of models, on both discrete and continuous state spaces. Classical proofs have long relied on delicate and model-specific computations, which typically required very precise information on the model at stake. Very recently, a more systematic  approach has emerged, based on an information-theoretic notion called \emph{varentropy}~\cite{Sal24,Sal25,PeSa25}. In order to successfully   predict cutoff,  this theory currently requires that the considered process have \emph{non-negative  Bakry-Emery curvature}, i.e.
\begin{equation}\label{Eq:NonNegCurv}
 \forall t\ge 0,\qquad	\Gamma P_t  \ \le\ P_t \Gamma,
\end{equation}
in the sense of quadratic forms, where $(P_t)_{t\ge 0}$ and $\Gamma$ denote  the semigroup and carré du champ operator, respectively. This condition  is well known to hold in many models, both in discrete and continuous spaces. For example, a Markov diffusion of the form
$$dX_t \ =\ -\nabla U(X_t)dt + dB_t\;,$$
is non-negatively curved if and only if the potential $U:\R^d\to\R$ is convex \cite{BGL14}. Similarly, a random walk on a  finite group $E$ is non-negatively curved as soon as its jump distribution $\nu\in\cP(E)$ is \emph{conjugacy-invariant}, in the sense that
\begin{equation}\label{Eq:ConjugInv}
 \forall x,y\in E,\qquad	\nu(xyx^{-1}) \ =\ \nu(y).
\end{equation}
An alternative approach was proposed in \cite{PeSa25b}, replacing varentropy estimates by transport cost estimates via Poincaré inequalities, still using non-negative curvature as a key tool. Recently, some other works \cite{Ped26,PedSal26} discuss other ways of relaxing the curvature assumption, and we shall discuss how they compare to our work below. 

The starting point of the present work is the observation that Assumption \eqref{Eq:NonNegCurv} fails in a number of interesting situations where one  still expects a cutoff, including:
\begin{enumerate}
    \item[(A)] elliptic diffusions in \emph{non-convex} potentials;
    \item[(B)] \emph{biased} random walks on groups;
    \item[(C)]  \emph{hypoelliptic} diffusions, notably kinetic diffusions.
\end{enumerate}Our goal here is to identify a suitable relaxation of Assumption \eqref{Eq:NonNegCurv}  which is simultaneously strong enough for the varentropy approach to carry over, and weak enough to  encompass a number of interesting examples from the classes (A), (B) and (C). \\ 

Specifically, we propose the following natural relaxation of   non-negative curvature, where a uniform pre-factor $M<\infty$ is introduced in the sub-commutation relation between the semi-group and carré du champ operator:
\begin{equation}\label{Eq:CurvatureC}
 \forall t\ge 0 \;,\quad	\Gamma P_t \ \le\ M P_t \Gamma \;.
\end{equation}
It turns out that  the theory presented in ~\cite{Sal25,PeSa25} extends almost verbatim  under  this  relaxed assumption. Moreover, despite the lack of an infinitesimal description (analogous to the  celebrated  $\Gamma_2\ge 0$ criterion  that so conveniently characterizes  non-negative  curvature), we will see that the optimal sub-commutation constant $M$ in  \eqref{Eq:CurvatureC}  can be effectively bounded using coupling techniques.   As a consequence, we deduce cutoff for a large class of elliptic diffusions in non-convex potentials (see Section \ref{sec:NonConvex}), and for a large class of  biased random walks on groups (see Section  \ref{sec:Groups}). The case of hypoelliptic diffusions (Section \ref{Sect_presentation_hypoelliptic}) is more subtle, and will require some additional regularization estimates.\\

The recent work \cite{PedSal26} by one of us and Pedrotti also tackles the problem of relaxing the non-negative curvature assumption, by replacing the relaxation time in the product condition by a non-equiibrium relaxation time, defined as the inverse \emph{local} Poincar\'e constant of the dynamics, uniformly in time and in the initial data. As mentioned in \cite[Remark 1]{PedSal26}, one way of controlling this non-equilibrium relaxation time is precisely to use generalized gradient estimates. While our results on elliptic diffusions could be deduced directly from \cite[Theorem 1]{PedSal26}, those concerning hypoelliptic processes do not fit in the framework of \cite{PedSal26}, as the main assumption therein fails when the Carr\'e du Champ is not coercive. 

Our relaxation \eqref{Eq:CurvatureC} is rich enough to encompass certain instances of hypoelliptic diffusions. For example, it was shown that the hypoelliptic diffusion on the Heisenberg group~\cite{DriMel,BBBC}, on SU(2)~\cite{BB} and on nilpotent Lie groups~\cite{Melcher} all satisfy \eqref{Eq:CurvatureC} for some constant $M\in(1,\infty)$. Notice however that all these examples satisfy the strong H\"ormander's condition, which has the following consequence: $\Gamma g =0$ implies $g$ constant.\\
On the other hand, as explained in the following remark, condition \eqref{Eq:CurvatureC} is bound to fail for all hypoelliptic diffusions which admit non-constant functions $g$ satisfying $\Gamma g = 0$. A prominent example is provided by kinetic diffusions.

\begin{remark}[Carr\'e du champ and relaxed curvature]\label{Rk:HypoCurv}
	If $\Gamma g=0$ then \eqref{Eq:CurvatureC} implies that $\Var g(X_t)=0$  for all $t\ge 0$, and since our process converges in law to its equilibrium measure, this forces $g$ to be constant $\mu$-almost everywhere. As a consequence, relaxed curvature is incompatible with the existence of a non-constant function $g$ satisfying $\Gamma g = 0$. This property is probably shared by all diffusions satisfying the weak H\"ormander's condition but not the strong one. Section \ref{Sect_presentation_hypoelliptic} will adress cutoff for certain processes that do not satisfy \eqref{Eq:CurvatureC}.
\end{remark}

In order to circumvent the above limitation, we work under a variant of \eqref{Eq:CurvatureC} which is specifically tailored to Markov diffusions, where the carré du champ operator $\Gamma$ is replaced with the squared norm of the gradient  in $\R^d$:
\begin{equation}\label{Eq:CurvatureNabla}
\forall t\ge 0,\;\quad	| \nabla P_t f|^2 \ \le\ M P_t|\nabla f|^2\;.
\end{equation}
Note that for elliptic diffusions, we have $\Gamma f=|\nabla f|^2$, and we are back to~\eqref{Eq:CurvatureC}. For hypoelliptic diffusions however, the two conditions differ genuinely. It turns out that nonetheless suitable gradient estimates have been established in the context of \emph{hypocoercivity}, which in combination with certain regularization estimates, will allow us to also cover certain hypoelliptic models. We will adapt the theory presented in~\cite{Sal25} to this new setting, see Theorem \ref{abstract_thm_cutoff_hypoelliptic}. 
As an application, we obtain the first proof of cutoff for kinetic mean-field diffusions (for nice enough interactions and confinement potentials). We will focus on kinetic diffusions in $\R^d$, but the approach should extend to the setting of sub-Riemannian manifolds.  

Another approach to extending the approach of \cite{Sal25} beyond non-negative curvature was proposed in \cite{Ped26}, relying on the Wasserstein distance contraction formulation of (coarse) curvature, but allowing for a small \emph{additive} error term (rather than the multiplicative, but not necessarily small, relaxation we consider here). We believe the approach we propose here fits better with various recent developments on hypocoercivity. In particular, it is not clear that the approach in \cite{Ped26} applies to kinetic diffusions. 

The sequel of this article is as follows: Section 2 presents the varentropy approach to cutoff, highlighting the elements we will modify here, as well as the statements of our new results. Section 3 contains the proofs, as well as applications to cutoff for mean-field dynamics and a discussion on an alternative approach to cutoff for diffusion processes via regularity estimates and transport maps. 

\paragraph{Acknowledgments.} We would like to warmly thank Paul Gassiat for identifying a mistake in the original version of the introduction and for pointing out the relationship between relaxed curvature and the strong and weak H\"ormander's conditions. The work of MF, CL and JS is supported by the Institut Universitaire de France (IUF). CL acknowledges financial support from ANR projects Smooth ANR-22-CE40-0017 and RANDOP ANR-24-CE40-3377. MF, AG and JS were supported by ANR project CONVIVIALITY ANR-23-CE40-0003. JS acknowledges support from the ERC grant CUTOFF (101123174). AG has benefited from a government grant managed by the Agence Nationale de la Recherche under the France 2030 investment plan ANR-23-EXMA-0001. 

\paragraph{Declaration on AI use.} AI was not used here: the ideas, the proofs and the writing are ours. 

\section{The varentropy approach to cutoff}\label{Sec:Varentropy}

Throughout the paper, we consider a continuous-time Markov process $(X_t)_{t\ge 0}$ taking values in a Polish space $E$ and approaching an invariant law  $\pi\in\cP(E)$ as $t\to\infty$. Its \emph{mixing time}  is defined as 
$$ \tmix(\varepsilon) \ :=\ \inf\{t\ge 0: \tv(X_t) \le \varepsilon \}\;,$$
where the parameter $\varepsilon\in(0,1)$ specifies the desired precision, and where  $\tv(Z) := \sup_{A\in\cB(E)} |\P(Z \in A) - \pi(A)|$ denotes the total-variation distance to equilibrium. Note that this definition is more flexible than the standard worst-case one, since the initial condition $X_0$ is here left as a tunable parameter of the model. This definition captures, in a very precise way, \emph{how long} the process  $(X_t)_{t\ge 0}$ needs to be run in order to approach equilibrium. Over the past decades, a rich variety of tools and methods have been developed to obtain  bounds on mixing times; see \cite{LPW,SalSF} for an introduction. 
A related, but more delicate question is to understand \emph{how abrupt} the transition from out of equilibrium to equilibrium is, i.e., to estimate the \emph{mixing width}
$$ \wmix(\epsilon)\ :=\ \tmix(\epsilon) - \tmix(1-\epsilon)\;,\quad \epsilon \in (0,1/2)\;.\\$$

In many probabilistic models, the Markov process under consideration depends on some natural parameter $n\in\N$, and one is interested in the regime where $n\to\infty$. In this asymptotic setup, a \emph{cutoff phenomenon} is said to occur when
\begin{eqnarray}
\label{Eq:Cutoff}
\forall\epsilon\in(0,1/2),\quad \frac{{\rm w}_{\textsc{mix}}^{(n)}(\varepsilon)}{{\rm t}_{\textsc{mix}}^{(n)}(\varepsilon)} & \xrightarrow[n\to\infty]{} & 0.
\end{eqnarray}
In concrete words, this means that instead of decreasing smoothly from $1$ to $0$, the distance to equilibrium approaches a step function as $n\to\infty$.\\

The  approach developed in~\cite{Sal24,Sal25,PeSa25} aims at producing  effective criteria for the occurrence of the cutoff phenomenon \eqref{Eq:Cutoff}, by providing systematic bounds on the mixing width of Markov processes. This is achieved by establishing an appropriate \emph{varentropy estimate}, as summarized in Theorem \ref{Th:Varentropy} below.  As usual, we write $(P_t)_{t\ge 0}$ for the semi-group, $\cL$ for its generator, and $\cE(\cdot,\cdot)$ for the the associated Dirichlet form:
\begin{eqnarray*}
\cE(f,g) & := & -\int\left(f\cL g\right){\rm d}\pi.
\end{eqnarray*} We define the variance, entropy, and varentropy of a probability density  $f$ on $(E,\pi)$ as
\begin{eqnarray*}
\Var(f) & := & \int (f-1)^2\, {\rm d}\pi\;,\\
\Ent(f) & := &  \int f\log f\,{\rm d}\pi\;,\\
\Varent(f) & := &  \int f(\log f)^2\, {\rm d}\pi-\left(\int f\log f\,{\rm d}\pi\right)^2\;.
\end{eqnarray*}
We also recall that the Poincaré constant $\lambda\ge 0$ and the modified log-Sobolev constant $\alpha\ge 0$ are the optimal constants in the respective functional inequalities
\begin{eqnarray}
\label{Eq:Poincaré}
 \lambda \Var(f) & \le & \cE(f,f) \\
 \label{Eq:mLSI}\alpha \Ent(f) & \le & \cE(f,\log f).
\end{eqnarray}
for all density functions $f$ on $(E,\pi)$ such that this makes sense. 
\begin{theorem}[The varentropy approach to cutoff, see {\cite[Lemma 5.4]{SalSF}}]\label{Th:Varentropy}
Suppose that for each $t>0$, $X_t$ admits a density $f_t$ w.r.t. $\pi$, and that the latter satisfies an estimate of the form
\begin{eqnarray}
\label{Eq:IDI}
\forall t>0,\qquad \Varent(f_t) & \le & \psi(t)\, \cE(f_t,\log f_t),
\end{eqnarray}
for some function $\psi\colon(0,\infty)\to(0,\infty)$. Then,  the mixing width satisfies
\begin{eqnarray}
\label{Eq:WmixLambdaAlpha}
\wmix(\varepsilon) \ \le \ \frac{1}{\lambda\varepsilon^3}+\sqrt{\frac{4m_\varepsilon}{\lambda\varepsilon^3}}\;,& \textrm{ and } &
\wmix(\varepsilon) \ \le \ \frac{1}{\alpha}\log\left(\frac{e\alpha(1+m_\varepsilon)}{2\varepsilon^4}\right)\;,
\end{eqnarray}
where $m_\varepsilon$ is the supremum of $\psi$ in the interval $[\tmix(1-\varepsilon),\tmix(\varepsilon)]$. 
\end{theorem}
What makes this result effective is the crucial observation that a varentropy estimate of the form \eqref{Eq:IDI} is guaranteed to hold as soon as the process $(X_t)_{t\ge 0}$ satisfies a rather standard concentration property known as a \emph{local Poincaré inequality}. To describe it, we assume to be given a rich enough algebra $\cA$ of test functions contained in the domain of $\cL$ and preserved under the action of $\cL$ and $(P_t)_{t\ge 0}$. The formula
\begin{eqnarray*} 
\Gamma(f,g) & := & \frac{1}{2}\left(\cL(fg) - f \cL g - g\cL f\right),
\end{eqnarray*}
then defines a  symmetric bilinear operator $\Gamma\colon\cA^2\to\cA$ called the  \emph{carré du champ}, which captures the infinitesimal evolution of variances along the Markov process, just as the generator does for expectations. Specifically, a classical computation~\cite[Eq. (3.1.20) and (3.1.21)]{BGL14} yields 
\begin{eqnarray*} 
\Var\left[f(X_t)\,|\,X_0\right] & = & 2\int_0^t (P_{t-s}\Gamma P_{s}f)(X_0){\,\rm d}s,
\end{eqnarray*}
for every $f\in\cA$ and $t\ge 0$, where  $\Gamma f:=\Gamma(f,f)$. Consequently, if our semi-group and carré-du-champ operator satisfy a mild \emph{sub-commutation relation} of the form 
\begin{eqnarray}
\label{Eq:SubCommutation}
\forall f\in\cA,\qquad\forall s\ge 0,\quad \Gamma P_s f & \le & \phi(s)\,P_s\Gamma f,
\end{eqnarray}
for some function $\phi\colon\R_+\to \R_+$, then we readily obtain  the \emph{local Poincaré inequality}
\begin{eqnarray}
\label{Eq:LocalPoincaré}
\Var\left[f(X_t)\,|\,X_0\right] \ \le \ \psi(t)\,\E\left[\Gamma f(X_t)\,|\, X_0\right], & \textrm{where} & \psi(t) \ := \ 2\int_0^t\phi(s){\,\rm d}s.
\end{eqnarray}
In particular, assuming that $X_0$ is non-random and that $\log f_t\in\cA$, we may safely take $f=\log f_t$ to conclude that
\begin{eqnarray}
\label{Eq:VarentGamma}
\Varent(f_t) & \le & \psi(t)\,\int f_t\Gamma \log f_t{\,\rm d}\pi.
\end{eqnarray}
In the smooth setting of Markov diffusions, the integral on the right-hand side is classically equal to $\cE(f_t,\log f_t)$, and \eqref{Eq:IDI} is established. On discrete state spaces, a similar relation has been shown to hold up to controllable error terms, as will be recalled below.\\

We are now left with the practical problem of establishing a sub-commutation relation of the form  \eqref{Eq:SubCommutation}, with an explicit and as small as possible  cost function $\phi$. Note that, by the semi-group property, the best possible $\phi$ automatically satisfies
\begin{eqnarray}
\label{Eq:SubMultiplicativity}
\phi(0)=1,\quad\textrm{ and }\quad \forall s,t\ge 0,\quad \phi(t+s) & \le & \phi(t)\phi(s).
\end{eqnarray}
Consequently, a uniform exponential estimate of the form
\begin{eqnarray}
\label{Eq:PhiExp}
\phi(t) & \le & e^{-2\kappa t},
\end{eqnarray}
is guaranteed to hold for all $t\ge 0$ as soon as it holds in an infinitesimal neighborhood around $t=0$. By an easy computation~\cite[Cor 3.3.19]{BGL14}, this reduces to the celebrated criterion
\begin{eqnarray}
\label{Eq:CDKappa}
\forall f\in\cA,\qquad \Gamma_2 f & \ge & \kappa \Gamma f,
\end{eqnarray}
where $\Gamma_2 f:=\frac{1}{2}\cL \Gamma f-\Gamma(f,\cL f)$ denotes the iterated carré du champ operator. The largest $\kappa\in[-\infty,\infty)$ for which this inequality holds is known as the Bakry-\'Emery curvature of the process, and the local nature of  \eqref{Eq:CDKappa} makes it easy to estimate in practice.  As soon as $\kappa\ge 0$, the varentropy estimate  (\ref{Eq:VarentGamma}) holds with $\psi(t)=2t$, which suffices to yield sharp estimates on the mixing window of many important models. This is, in a nutshell, the content of the works~\cite{Sal25,PeSa25}. \\

The starting point of the present work is the simple observation that the infinitesimal behavior of the sub-commutation cost $t\mapsto \phi(t)$ near $t=0$ could, in some cases, provide a very pessimistic estimate on its large-time behavior. For example, the function  $\phi\colon t\mapsto e^{\sqrt{t}-2t}$ clearly satisfies (\ref{Eq:SubMultiplicativity}), but the expansion $\phi(t)=1+\sqrt{t}+o(\sqrt{t})$ at $t=0$ is clearly incompatible with any    estimate of the form \eqref{Eq:PhiExp}, unless $\kappa=-\infty$. Clearly, it would be much better here to use instead the estimate $\phi(t)\le e^{1-t}$, where the bad behavior near $t=0$ has been simply absorbed into a  pre-factor whose impact on the final varentropy estimate \eqref{Eq:VarentGamma} is harmless. In light of this idealized example, we will focus on the following sub-commutation relation:
\begin{eqnarray*}
\forall f\in\cA,\quad \forall t\ge 0,\quad\Gamma P_t f & \le & Me^{-2\kappa t} P_t\Gamma f.
\end{eqnarray*}
Note that we must necessarily have $M\ge 1$, with equality corresponding to the Bakry-\'Emery definition of curvature bounded from below by $\kappa$. In view of \eqref{Eq:LocalPoincaré}, this relaxed notion guarantees the same local Poincaré inequality as the original one, up to the multiplicative constant $M$. 
Of course, the real price to pay is the lack of a convenient infinitesimal characterization of the form \eqref{Eq:CDKappa}. Nevertheless, we will see that effective bounds on the parameters $(M,\kappa)$ can easily be obtained in practice using couplings, and that this substantially broadens the class of processes for which a cutoff phenomenon can be predicted. Let us start with the important case of elliptic diffusions, for which the situation is particularly pleasant.

\subsection{Elliptic diffusions in non-convex potential}\label{sec:NonConvex}

Our first contribution is the following extension of the main result in~\cite{Sal25}, which corresponds to the special case $M=1$.

\begin{theorem}[Cutoff for diffusions]\label{Th:Relaxed}
Assume the following:
\begin{enumerate}
\item (Chain rule) The following identity holds generically:
\begin{equation}\label{Eq:ChainRule}
{\cL e^f} \ = \ e^f\left(\cL f + \Gamma  f\right)\;.
\end{equation}
\item (Spectral gap) The Poincar\'e constant $\lambda$ defined in~\eqref{Eq:Poincaré} is strictly positive.
\item (Approximate curvature) There is $M<\infty$ such that, generically:
\begin{equation}\label{Eq:RelaxedCurvature}
\forall t\ge 0,\qquad \Gamma P_t f \ \le \ M\, P_t\Gamma f \;.
\end{equation}
\end{enumerate}
Then, from any deterministic initialization, we have the  mixing-window estimate
	$$ \forall \epsilon\in\left(0,\frac 12\right),\qquad \wmix(\epsilon) \ \le \ \frac{1}{\lambda \epsilon^3} + \sqrt{\frac{8M\tmix(\epsilon)}{\lambda \epsilon^3}}\;.$$
In particular, a family of such processes exhibits cutoff as soon as
	$$ \exists \epsilon \in (0,1)\;,\quad \frac{\lambda_n {\rm t}_{\textsc{mix}}^{(n)}(\epsilon)}{M_n} \underset{n\to\infty}{\longrightarrow} +\infty\;.$$
\end{theorem}

Let us make a few comments on the assumptions. First, thanks to the usual chain rule for differentiating composed functions, the identity \eqref{Eq:ChainRule} holds as soon as $\cL$ is a second-order differential operator acting on a space of smooth functions, as is classically the case for solutions to stochastic differential equations on Euclidean spaces or manifolds.
Second, the positivity of the Poincaré constant $\lambda$ yields the classical bound
\begin{equation}\label{Eq:SpectralGap}
	\Var_\mu(P_t^\star f)\ \le \ \Var_\mu(f) e^{-2 \lambda t}\;.
\end{equation}
Finally and perhaps more importantly, the requirement \eqref{Eq:RelaxedCurvature} is a substantial relaxation of the classical non-negative curvature condition of Bakry and \'Emery \cite{BE85}, which corresponds to the case $M=1$. In the context of diffusions on Euclidean spaces (or smooth manifolds) with volatility equal (or at least proporitional) to the identity matrix, we have $\Gamma(f)=  |\nabla f|^2$ and a general duality result due to Kuwada \cite{Kuw10} guarantees that  \eqref{Eq:RelaxedCurvature}  is equivalent to the geometric property
\begin{equation}\label{Eq:Contraction}
\forall t\ge 0,\qquad W_2(\mu P_t,\nu P_t) \ \le \ \sqrt{M}W_2(\mu,\nu)\;,
\end{equation}
for all $\mu,\nu\in\cP(E)$, where $W_2(\mu,\nu):=\inf_{X\sim \mu, Y\sim\nu}\sqrt{\mathbb E[|X-Y|^2]}$ is the $2-$Wasserstein distance between $\mu$ and $\nu$. 
This provides an effective, coupling-based method for verifying  \eqref{Eq:RelaxedCurvature} on concrete models. A simple generic example, studied for example in \cite{Mon23}, is the Langevin diffusion whose potential fails to be convex in a compact region. 

\begin{corollary}[Langevin dynamics in non-convex potential] \label{Cor:elliptic_high_temp}
	Let $(X_t)_{t\ge 0}$ solve
	$$dX_t \ =\ -\nabla U(X_t)dt + \sqrt{2T}dB_t\;,$$
	on $E=\R^d$, with $B$ a $d-$dimensional Brownian motion, $X_0$ a deterministic initialization, $T>0$ a parameter, and $U\colon \R^d\to\R$ a smooth potential satisfying for some $R,K,c,\delta \in(0,\infty)$,
	\begin{enumerate}
		\item $\nabla^2 U(x) \ge c \mathrm{Id}$ when $|x|>R$;
		\item $-K Id \le \nabla^2 U(x)\le K Id$ on $\R^d$;
		\item $|\nabla U(0)| \leq \delta$.
	\end{enumerate} 
	Then, there is $C<\infty$ depending only on $(K,c)$, such that whenever $T\ge C(R^2+\delta R)/d$,
	$$ \forall \epsilon\in\left(0,\frac 12\right),\qquad \wmix(\epsilon) \ \le \ \frac{32 T}{c \epsilon^3} + 16\sqrt{\frac{2T \tmix(\epsilon)}{c\epsilon^3}}\;.$$
	In particular, any sequence of diffusions satisfying the above assumptions for some fixed triple $(K,c,T)$ and variable $R_n, \delta_n, d_n$ exhibits cutoff as soon as
	\begin{enumerate}
	\item $\max(R_n,\delta_n) \le C_0 \sqrt{d_n}$ for some constant $C_0$,
	\item $T \ge 2 C C_0^2$,
	\item $ \exists \epsilon \in (0,1)$, ${{\rm t}_{\textsc{mix}}^{(n)}(\epsilon)} \ \underset{n\to\infty}{\longrightarrow} \  +\infty$.
	\end{enumerate}
\end{corollary}

To illustrate this result, let us present two examples.

\begin{example}[Fixed dimension and diverging initial condition]
	We consider a fixed diffusion in $\R^d$ and pick a sequence of initial conditions $(x_n)_n$ that diverges. We take $U = V +W$ where $V$ is quadratic and $W$ is a smooth, compactly supported function. Since no additional hypothesis is made of $W$, $U$ is not convex in general. However the assumptions made on $U$ in Corollary \ref{Cor:elliptic_high_temp} are satisfied. Now take a sequence $(x_n)_n$ s.t.~$|x_n|\to\infty$ as $n\to\infty$. Since the potential is quadratic at infinity, the mixing time is asymptotically of order $\log n$ so our corollary ensures that cutoff occurs.\\
	It is interesting to note that if $V$ is not quadratic but grows like a polynomial of degree strictly larger than $2$, then one assumption of Corollary \ref{Cor:elliptic_high_temp} is not met; and at the same time the mixing times remain bounded with $n$ and cutoff does not occur in that setting.  
\end{example}

\begin{example}[Diverging dimension]
	In dimension  $d_n = n$, consider the following potential, which is a regularization of the more natural quantity $(|x|-\sqrt{n})^2/2$. 
	$$U_n(x) = \frac{1}{2}|x|^2 -2\sqrt{n}(n+|x|^2)^{1/2}.$$
	Then, Corollary \ref{Cor:elliptic_high_temp} applies with $K = 1$, $\delta = 0$, $c = 1/2$ and $R = \sqrt{15n}$.  
\end{example}

Note that $\sqrt{d_n}$ is a natural scale for the variable $R_n$ and $B_n$, under which the required lower bound on $T_n$ is uniform. We could have stated a version with relaxed assumptions (in particular, we can easily weaken the assumption of uniform  upper bound on $\nabla^2U$, as shall be seen in the proof), but we chose to restrict our attention to this simple setting to obtain a more compact formulation. Similarly, the condition of convexity outside a ball can be relaxed by  a more careful use of Theorem \ref{thm_abstrait_monmarche}. 

\subsection{Biased random walks on finite groups}\label{sec:Groups}
Let us now consider the case where the state space $E$ is finite, as considered in \cite{PeSa25}. Upon a trivial time rescaling, our Markov generator $\cL$ then takes the form
\begin{equation}\label{Eq:DiscreteGenerator}
\forall x\in E,\qquad (\cL f)(x) \ := \ \sum_{y\in E}T(x,y)\left(f(y)-f(x)\right)\;,
\end{equation}
for all $f\colon E\to\R$, where $T$ is a stochastic matrix that specifies the transition probabilities. The associated carré du champ is given by
\begin{equation}\label{Eq:DiscreteGamma}
\forall x\in E,\qquad (\Gamma f)(x) \ := \ \frac{1}{2}\sum_{y\in E}T(x,y)\left(f(y)-f(x)\right)^2\;.
\end{equation}
Although the chain rule \eqref{Eq:ChainRule} notoriously fails in this context, a suitable quantitative approximation of it was established in \cite{PeSa25}, under the mild symmetry assumption
\begin{equation}\label{Eq:Symmetry}
\forall (x,y)\in E^2,\qquad T(x,y) \ > \ 0 \ \Longrightarrow \ T(y,x)\ >\ 0\;.
\end{equation}
This was used in \cite{PeSa25} to produce an effective cutoff criterion for non-negatively curved Markov chains. As we will see, the proof extends verbatim under our relaxed notion of curvature, yielding the following generalization of the main result in \cite{PeSa25}. 
\begin{theorem}[Cutoff for discrete chains]\label{Th:RelaxedDiscrete}
Consider a finite-state space Markov chain with symmetric support as above. Assume that there are $d,M,\alpha\in(0,\infty)$ such that:
\begin{enumerate}
	\item (Sparsity) for every \(x\ne y\), \(T(x,y)>0\) implies \(T(x,y)\ge d^{-1}\);
	\item (Entropy decay) The modified log-Sobolev constant defined in~\eqref{Eq:mLSI} is at least $\alpha > 0$;
	\item (Approximate curvature) for every \(f:E\to\R\) and $t\ge 0$,
	\begin{equation}\label{Eq:RelaxedCurvatureDiscrete}
	\Gamma P_t f\le M\,P_t\Gamma f.
	\end{equation}
\end{enumerate}
Then, from any deterministic initial condition, and for every \(\varepsilon\in(0,1/2)\), we have
\[
\wmix(\varepsilon)
\le
\frac{c_\varepsilon}{\alpha}
\left(
1+\log M+\log(1+\log d)
+\log\bigl(1+\alpha\tmix(\varepsilon)\bigr)
\right),
\]
where \(c_\varepsilon<\infty\) depends only on \(\varepsilon\).  In particular,  any sequence of such chains exhibits cutoff as soon as, for some \(\delta\in(0,1)\),
\begin{equation}\label{Eq:CutoffDiscrete}
\frac{\alpha_n{\rm t}_{\textsc{mix}}^{(n)}(\delta)}
{1+\log M_n+\log(1+\log d_n)}
\longrightarrow+\infty.
\end{equation}
\end{theorem}

To appreciate the gain offered by our relaxed curvature condition, let us consider the important special case of random walks on groups. More precisely, assume that $E$ is a finite group and that $\cL$ is of the form
\begin{equation}
\forall x\in E,\qquad \cL f(x) \ := \ \sum_{z\in E}\nu(z)\left(f(zx)-f(x)\right)\;,
\end{equation}
where $\nu\in\cP(E)$ is a probability measure whose support generates the group. In this context, we will see that the approximate curvature condition~\eqref{Eq:RelaxedCurvatureDiscrete} holds as soon as the jump distribution $\nu$ is approximately conjugacy-invariant, in the following sense:
\begin{equation}
\forall x,y\in E,\qquad \nu(yxy^{-1})\ \le\ M\nu(x)\;.
\end{equation}
To give a concrete example, consider the symmetric group $\mathfrak S_n$ with a jump distribution $\nu_n$ which is supported on transpositions. When $\nu_n$ is uniform, the occurrence of a cutoff as $n\to\infty$ is a celebrated historical result proved by Diaconis and Shahshahani using representation theory \cite{DiSh81}. Surprisingly, extending their approach to the case where $\nu_n$ takes two distinct values turned out to  constitute a remarkable tour-de-force, achieved only recently in \cite{NeYa24}. In contrast, our approach allows for an arbitrary number of values, and requires nothing more than a classical modified logarithmic Sobolev inequality \eqref{Eq:mLSI} under the uniform law \cite{BoTe06,GaQu03,Goe04}. 
\begin{corollary}[Biased transpositions on the symmetric group] \label{Cor:rw_symmetric_group}
For each $n\ge 1$, consider a random walk on the symmetric group ${\mathfrak S}_n$ whose jump distribution $\nu_n$ is supported on the set of transpositions  ${\mathcal T}_n$. Then a cutoff occurs as $n\to\infty$, as long as 
\begin{equation}
\sup_{n\ge 1} \left\{\frac{\max_{\tau\in {\mathcal T}_n}\nu_n(\tau)}{\min_{\tau\in{\mathcal T}_n}\nu_n(\tau)}\right\} \ < \ \infty.
\end{equation}
\end{corollary}
We emphasize that this is just one particular illustration of the above theorem, chosen for the sake of concreteness, among several possible alternatives. For example, the same conclusion holds (with the exact  same argument) if the set of transpositions ${\mathcal T}_n$ is replaced with another conjugacy class, as long as  the number $k_n$ of non-fixed elements ($k_n=2$ in the case of transpositions) satisfies $k_n=n^{o(1)}$ as $n\to\infty$. 

\subsection{Hypoelliptic diffusions} \label{Sect_presentation_hypoelliptic}

Hypoelliptic diffusions on $\R^d$ are stochastic processes of the form 
$$dZ_t = b(Z_t)dt + \sigma(Z_t)dB_t$$
where the matrix $\sigma(x)$ has a non-trivial kernel, but where the drift $b(x)$ nonetheless propagates the noise onto the remaining directions. The prototypical example (the only one we shall discuss in details here) is a kinetic diffusion $Z_t = (X_t,V_t) \in \R^d\times \R^d$ where
$$dX_t = V_tdt\;, \hspace{2mm} dV_t = -\nabla U(X_t)dt - V_tdt + \sqrt{2}dB_t.$$
This corresponds to damped Hamiltonian dynamics with additive noise. Here the noise  acts directly on the velocity variable $V$, but propagates onto the position variable $X$ via the Hamiltonian part.  Such processes are highly relevant to modern problems in MCMC simulations, since they allow for accelerated sampling procedures  compared with the more standard overdamped dynamics \cite{Lu25}. Such hypoelliptic diffusions also naturally appear in geometry, as analogues of Brownian motions for sub-Riemannian manifolds \cite{Bau18LN}. 

These processes do not fit in the framework of \cite{Sal25} for two reasons: their curvature in the sense of Bakry-Emery cannot be bounded from below, and moreover they do not satisfy the differential relation between entropy and Fisher information that holds for elliptic diffusion processes. Both issues arise from the fact that the carr\'e du champ operator is \emph{degenerate}, and the associated Dirichlet form not coercive. For example, for the kinetic Fokker-Planck equation, we have
$$\cE(f,f) = \int{|\nabla_v f(x,v)|^2d\mu},$$
which is zero as soon as $f$ depends on the position variable $x$ but not on the velocity variable $v$. Hypocoercive techniques  going back to \cite{HeNi04,Vil09} have been developed to overcome this degeneracy. See for example \cite{Her18} for a short account, as well as the introduction of \cite{BFLS22} and the lecture notes \cite{Vil26} for a survey. We will use ideas from this area to prove cutoff for certain families of kinetic diffusions. Our main abstract result, Theorem \ref{abstract_thm_cutoff_hypoelliptic} below, identifies certain regularization estimates that allow to bound the cutoff window, leading to the following concrete criterion for cutoff of kinetic dynamics: consider a sequence of diffusion processes on $\R^{d_n}\times \R^{d_n}$ of the form 
\begin{equation} \label{eq__kin_dynamic}
    \begin{cases}
    dX^n_t &= V^{n}_tdt \\
    dV^n_t &= -\nabla U^n(X^n_t)dt -  V^n_tdt + \sqrt{2}dB^n_t,
    \end{cases}
\end{equation}
where $U^n$ is normalized so that $e^{-U^n}$ is a probability density (we shall always enforce assumptions that ensure this measure has finite mass). The invariant measure of (\ref{eq__kin_dynamic}) is then $e^{-U^n}\otimes \gamma$, where $\gamma$ is the standard Gaussian law.

\begin{theorem}[Cutoff for kinetic dynamics] \label{thm_kinetic_cutoff_general}
Consider a sequence of dynamics of the form \eqref{eq__kin_dynamic} with the dimension $d_n$ going to infinity. Assume that the initialization is deterministic, and write $R_n$ for its distance to the center of mass $m_n := \int{xe^{-U^n(x)}dx}$. Assume moreover that there exist constants $0 < m \leq M$ such that for all $n$
$$mI_{d_n} \leq \nabla^2U^n \leq M I_{d_n}\;, \hspace{3mm} \sqrt{M}-\sqrt{m} < 1.$$
Then cutoff occurs as soon as ${\rm t}_{\textsc{mix}}^{(n)} \gg \log \log (d_n\vee R_n)$. 
\end{theorem}

This is  not the only class of potentials to which Theorem \ref{abstract_thm_cutoff_hypoelliptic} applies (see for example Remark \ref{rmk_hypo_monmarche}): we only wanted to highlight one  concrete situation where the assumptions are relatively easy to state. As an application, we establish the first occurence of a cutoff phenomenon for kinetic mean-field dynamics. 

\begin{corollary}[Non-quadratic convex interactions] \label{Cor:Cutoff_kinetic_mf}
Let \(p\ge1\) be fixed.  For \(n\ge1\), consider on \((\R^p)^n\times(\R^p)^n\) the kinetic dynamics
\[
    \begin{cases}
    d X_t^i = V_t^i\,d t,\\[2mm]
    d V_t^i =
    -\nabla_{x_i}U_n(X_t)\,d t
    -V_t^i\,d t
    +\sqrt2\,d B_t^i,
    \end{cases}
    \qquad 1\le i\le n,
\]
where
\[
    U_n(x_1,\ldots,x_n)
    =
    \sum_{i=1}^n \hat{U}(x_i)
    +
    \frac1{2n}\sum_{i,j=1}^n W(x_i-x_j).
\]
Assume that \(\hat{U},W\in C^2\), that \(W\) is even, and that there exist constants \(0<m_{\hat{U}}\le M_{\hat{U}}<\infty\) and \(L_W<\infty\), independent of \(n\), such that
\[
    m_{\hat{U}} I_p\le \nabla^2 \hat{U} \le M_{\hat{U}} I_p,
    \qquad
    0\le \nabla^2 W\le L_W I_p.
\]
Then, under the perturbative condition
\[
    \sqrt{M_{\hat{U}}+L_W}-\sqrt{m_{\hat{U}}}<1,
\]
 the hypocoercive \(W_2\)-contraction constants are uniform in \(n\).  In particular, the abstract cutoff criterion applies uniformly to this family: if $R_n$ denotes the distance of the deterministic initial condition $(x_{i,n},v_{i,n})_{i=1,\ldots,n}$ to the center of mass then cutoff occurs as soon as
\[
{\rm t}_{\textsc{mix}}^{(n)}\gg \log\log (n \vee R_n).
\]
\end{corollary}

Note that our assumptions ensure that the confinement potential is convex, so we do not need any high-temperature assumption. The mixing-time condition is easily deduced from known results on propagation of chaos for mean-field dynamics, leading to the following corollary.

\begin{corollary} \label{Cor:Cutoff_kinetic_mf_application}
In the context of the previous theorem, assume that $\hat{U}(x) = (M/2) |x|^2$ for some constant $M>0$, and that $W$ is even and satisfies $0\le \nabla^2 W\le L_W I_p$ with $0 \le L_W < M / ( 4 M + 2)$. Assume further that the initial condition $(x_{i,n},v_{i,n})_{i=1,\ldots,n}$ satisfies $x_{1,n}=\ldots=x_{n,n}$ and $v_{1,n}=\ldots=v_{n,n}$, and that there exists $r>0$ such that $|x_{i,n}| \le r$ and $|v_{i,n}| \le r$ for all $n$. Then cutoff occurs as $n \longrightarrow +\infty$.
\end{corollary}

\begin{remark} \label{Rem:initial_data_mf}
Assuming  that the particles all start from the same position and velocity conveniently allows us to export known results from \cite{CLRW24} on propagation of chaos, that require i.i.d. initial data (at least, to get suitable bounds without new effort). It is likely that this can be relaxed with some extra work on quantitative propagation of chaos for kinetic dynamics, but this is beyond our scope here. 
\end{remark}

The assumption that $U$ is quadratic can be relaxed, as we shall discuss in Remark \ref{rem_non_quadratic_confinement}. 

\section{Proofs}

We start with some generic definitions and estimates used throughout the article. In the sequel, $\Ent(t)$ (resp. $\Varent(t)$) denotes the relative entropy (resp. varentropy) of $f_t$, the density at time $t$ w.r.t. the invariant probability measure $\pi$. We also define
\begin{equation}\label{Eq:IGamma}
	I^\Gamma(t) := \int \Gamma(\log f_t) f_t d\pi\;.
\end{equation}
On $\R^d$ and provided $\Gamma f = |\nabla f|^2$ (up to a multiplicative constant), this coincides (up to a multiplicative constant) with the usual Fisher information
\begin{equation}\label{Eq:I}
	I(t) := \int \frac{|\nabla f_t|^2}{f_t} d\pi\;,
\end{equation}
but in general (in particular for hypoelliptic diffusions) it might differ from the latter. Let us now collect a few estimates on the relative entropy. 

First of all, by the reverse Pinsker's inequality~\cite[Lemma 1]{Sal25}, we have
\begin{equation}\label{Eq:ReversePinsker}
	t \geq \tmix(1-\epsilon)\ \Longrightarrow \ \Ent(t) \leq \frac{1 + \sqrt{\Varent(t)}}{\epsilon}\;.
\end{equation}

Second, we have the following identities:
\begin{lemma}
	Generically
	\begin{equation}\label{Eq:genericDerivEnt}
		\frac{d}{dt}\Ent(t) = \int f_t \cL \log f_t d\pi\;.
	\end{equation}
	Furthermore, if the chain rule \eqref{Eq:ChainRule} holds then
	\begin{equation}\label{Eq:DerivEnt}
		\frac{d}{dt}\Ent(t) = -\int \Gamma(\log f_t) f_t d\pi = - I^\Gamma(t).
	\end{equation}
\end{lemma}
	In the setting of diffusion processes, the second identity, or its equivalent version with $\Gamma(\log f_t) f_t$ replaced by $\Gamma(f_t) / f_t$, is standard~\cite[Prop 5.2.2]{BGL14}. We present a proof of~\eqref{Eq:DerivEnt} that only uses the specific chain rule~\eqref{Eq:ChainRule} (and not the generic chain rule satisfied by diffusion processes \cite[Definition 1.11.1]{BGL14}).
\begin{proof}
	 We denote by $\cL^*$ the adjoint in $L^2(d\pi)$ of $\cL$. Since $\partial_t f_t = \cL^* f_t$ we obtain
	\begin{align*}
		\frac{d}{dt}\Ent(t) =\frac{d}{dt}  \int f_t \log f_t d\pi = \int (\cL^* f_t)(1+ \log f_t) d\pi = \int (\cL^* f_t) \log f_t d\pi,
	\end{align*}
	since $\pi$ is an invariant measure. Then we write
	\begin{align*}
		\int (\cL^* f_t) \log f_t d\pi =\int f_t (\cL \log f_t) d\pi,
	\end{align*}
	which is the first identity of the statement.
	To get the second, we use once again the fact that $\pi$ is an invariant measure, and we rely on the chain rule~\eqref{Eq:ChainRule} to finally get
	\begin{align*}
		\int f_t (\cL \log f_t) d\pi = \int \big(f_t (\cL \log f_t) - \cL f_t\big) d\pi = - \int \Gamma(\log f_t ) f_t d\pi,
	\end{align*}
	as desired.
\end{proof}

We will also use the following mixing-time estimate from~\cite[Lemma 2]{Sal25}:
\begin{equation}\label{Eq:MixEnt}
		\tmix(\epsilon) \leq t + \frac{1 + \Ent(t)}{\lambda \epsilon}\;.
\end{equation}

\subsection{Elliptic diffusions at high temperature}\label{Sec:Elliptic}

\subsubsection{Proof of Theorem \ref{Th:Relaxed}}

The proof very closely follows \cite{Sal25}, so for the parts that are left unchanged we only give a sketch. We set $t_0 = \tmix(1-\epsilon)$. 
By the reverse Pinsker inequality~\eqref{Eq:ReversePinsker}, we have
$$t \geq t_0 \ \Longrightarrow \ \Ent(t) \leq \frac{1 + \sqrt{\Varent(t)}}{\epsilon}.$$
The arguments that led to~\eqref{Eq:VarentGamma} were only based on the approximate curvature bound~\eqref{Eq:RelaxedCurvature}, so that we can write
$$\Varent(t) \leq 2Mt I^\Gamma(t).$$
Therefore
$$t \geq t_0 \ \Longrightarrow \ \Ent(t) \leq \frac{1+ \sqrt{2Mt I^\Gamma(t)}}{\epsilon}$$
which, combined with~\eqref{Eq:DerivEnt}, implies that while $t \geq t_0$ and $\Ent(t) > \epsilon^{-1}$, 
$$\frac{d}{dt}\Ent(t) \leq -\frac{(\epsilon \Ent(t)-1)^2}{2Mt}.$$
Arguing as in \cite{Sal25}, solving this differential inequality yields
$$t \geq t_0 \ \Longrightarrow \ \Ent(t) < \frac{1}{\epsilon} + \frac{2M}{\epsilon^2 \log(t/t_0)}.$$
From the mixing-time estimate~\eqref{Eq:MixEnt} we deduce that for any $t > t_0$
\begin{align*}
\tmix(\epsilon) &\leq t + \frac{1}{\lambda \epsilon} + \frac{1}{\lambda \epsilon^2} + \frac{2M}{\lambda \epsilon^3\log(t/t_0)} \\
&\leq t +  \frac{1}{\lambda \epsilon^3} + \frac{2Mt_0}{\lambda \epsilon^3(t-t_0)}
\end{align*}
where we used $M \geq 1$ and $\epsilon \leq 1/2$. Choosing $t= t_0 + \sqrt{\frac{2Mt_0}{\lambda\epsilon^3}}$ (which optimizes the above upper bound) concludes the proof. 

\subsubsection{Proof of Corollary \ref{Cor:elliptic_high_temp}}

We now concentrate on diffusions in $\R^d$ taking the form
$$dX_t = b(X_t)dt + \sqrt{2T}dB_t\;,$$
for some function $b:\R^d \to \R^d$. We denote by $(P_t)_{t\ge 0}$ its semigroup and by $(P_t^*)_{t\ge 0}$ the dual semigroup, acting on measures. Observe that $\Gamma f = T |\nabla f|^2$.\\

Before establishing Corollary \ref{Cor:elliptic_high_temp}, that concerns the specific situation where $b=-\nabla U$, we start by collecting results that hold without this specific assumption (in particular, reversibility will not be required).

\begin{definition}
Let $\varphi : \R_+ \longrightarrow \R_+$ be a non-increasing function. The semigroup is said to be a Wasserstein contraction with rate $\varphi$ if for any $\mu, \nu \in \mathcal{P}(\R^d)$ we have
$$W_2(P_t^*\mu, P_t^*\nu) \leq \varphi(t)W_2(\mu, \nu) \hspace{3mm} \forall t \geq 0.$$
\end{definition}

In the reversible case where $b=-\nabla U$, the semigroup is a Wasserstein contraction with rate $\varphi(t)= e^{-\lambda t}$ if and only if $U$ is $\lambda$-convex \cite{StVRe05}. Proving such estimates with $\varphi(t)> 1$ in more general settings is one of the applications of hypocoercive techniques \cite{Bau17,Mon23}. The main advantage of this notion is that it is invariant by equivalent change of distances (up to the value of the function $\varphi$). The interest of looking at formulations of rates of decay that allow for an equivalent change of metric was pointed out in \cite{Vil09}, and was exploited in various forms in the literature, among which \cite{Ebe16,EGZ19,BGHKM25,LMM25}. 

It turns out that such Wasserstein distance contraction estimates are \emph{equivalent} to the gradient estimates we are looking for: 

\begin{proposition}[Kuwada duality, \cite{Kuw10}]\label{Prop:Kuwada}
Let $\varphi : \R_+ \longrightarrow \R_+$ be a non-increasing function. The following are equivalent : 
\begin{enumerate}
    \item We have $|\nabla P_t f|^2 \leq \varphi(t)^2P_t(|\nabla f|^2)$ for all $t \geq 0$ and $f$. 
    
    \item The process is a Wasserstein contraction with rate $\varphi$. 
\end{enumerate}
\end{proposition}
This statement is a particular case of a broader duality theorem allowing for $L^p$ Wasserstein distances, but we shall only use the quadratic case here. For our purpose, the main application of this more general form of Wasserstein contraction is that it still implies a local Poincar\'e inequality: 

\begin{proposition}[Local Poincar\'e inequality for Wasserstein contractions] \label{loc_pi_elliptic}
	Let $\varphi : \R_+ \longrightarrow \R_+$ be a non-increasing function and assume that the semigroup is a Wasserstein contraction with rate $\varphi$. Let $\psi(t) = 2\int_0^t{\varphi(s)^2ds}$. 
Then a local Poincar\'e inequality holds : for all $t > 0$, $x \in \R^d$ and locally Lipschitz function $f$
$$P_t(f^2)(x) - (P_t f(x))^2 \leq \psi(t)P_t(|\nabla f|^2)(x) $$
\end{proposition}
\begin{proof}
	The carr\'e du champ operator is given by $\Gamma f = T |\nabla f|^2$. From the Proposition \ref{Prop:Kuwada}, we deduce that
	$$ \Gamma P_t f \leq \varphi(t)^2 P_t \Gamma f\;,\quad \forall t\ge 0\;.$$
	We can thus follow the arguments presented above~\eqref{Eq:LocalPoincaré} and deduce that
\begin{eqnarray*}
P_t(f^2)(x) - (P_t f(x))^2 & = & \Var\left[f(X_t)\,|\,X_0=x\right] \\
& \le & \psi(t)\,\E\left[\Gamma f(X_t)\,|\, X_0=x\right] = \psi(t)P_t(|\nabla f|^2)(x)\;.
\end{eqnarray*} 
\end{proof}

$L^1$ Wasserstein contractions have also been used in the context of cutoff (the so-called Ollivier-Ricci curvature condition), but they do not imply a local Poincaré inequality by themselves, and hence in the approach of \cite{Sal24} must be used in conjunction with other estimates on varentropy.  
The following result is one of the conclusions in \cite[Theorem 2]{Mon23}, allowing to control the relaxation time, that is, the inverse of the Poincar\'e constant which was defined in~\eqref{Eq:Poincaré}.

\begin{lemma}[$L^2$ hypocoercivity] \label{Lem:relaxation_time_by_contraction}
Assume that the semigroup is a Wasserstein contraction with $\varphi(t) = \sqrt{M}e^{-\varrho t}$, $t\ge 0$, for some constants $M,\varrho > 0$. Then
$$\frac1{\lambda} \leq \frac{MT}{\varrho}\;.$$
\end{lemma}

To justify this framework we need concrete non-convex examples where Wasserstein contraction holds. The following is due to Monmarch\'e \cite[Theorem 1]{Mon23}: 
\begin{theorem} \label{thm_abstrait_monmarche}
Consider a diffusion process
$$dX_t = b(X_t)dt + \sqrt{2T}dB_t$$
for some function $b:\R^d \to \R^d$ and let 
$$k(x) := - \sup_{y \neq x} \left\{\frac{\langle (x-y), b(x)-b(y)\rangle}{|x-y|^2} \right\}$$
Assume that there exist $K, R_0 \geq 0$ and $c_0> 0$ such that
$$k(x) \geq - K \hspace{1mm} \forall x \in \R^d\;, \hspace{3mm} k(x) \geq c_0 \hspace{1mm} \text{ if } |x| \geq R_0.$$
Let $R_* := R_0(2 + 2K/c_0)^{1/d}$. Then for 
$$T \geq T_0 := (2K+c_0)\frac{(2K+c_0/2)R_*^2 + 2\sup_{|x| \leq R_*} \{-x \cdot b(x)\}}{c_0d},$$
Wasserstein contraction holds (with respect to the usual Euclidean distance) with $\varphi(t) = \sqrt{M}e^{-\varrho t}$ where
$$M := 1+\frac{(2K+c_0)R_*^2}{2dT}\;, \hspace{3mm} \varrho := \frac{c_0}{4}.$$
\end{theorem}
The dimension appears in the expressions. However, for typical examples $R$ is of order $\sqrt{d}$ while $K$ and $c$ are dimension-free, which for Lipschitz drifts $b$ leads to dimension-free admissible values for $T_0$, $M$ and $\lambda$.  The proof of this theorem uses a reflection coupling, building upon work of Eberle \cite{Ebe16}. The use of reflection couplings and contraction estimates to prove local Poincar\'e inequalities was pointed out in \cite{CaGu14}. 

\begin{proof} [Proof of Corollary \ref{Cor:elliptic_high_temp}]
	Recall that we deal with the diffusion process
	$$ dX_t \ =\ -\nabla U(X_t)dt + \sqrt{2T}dB_t\;,$$
	so that $\Gamma f = T |\nabla f|^2$. Assume for now that, provided $T\ge C(K,c) \frac{R^2 + \delta R}{d}$ for some constant $C(K,c)$ only depending on $K$ and $c$, the associated semigroup is a Wasserstein contraction with $\varphi(t) = \sqrt{M} e^{-\varrho t}$ where
	$$ M < 2\;,\quad \varrho := \frac{c}{16}\;.$$
	We aim at applying Theorem \ref{Th:Relaxed}, let us check that the required assumptions are satisfied. The chain rule holds in our setting since we deal with a diffusion process. The Poincaré constant is strictly positive since, by Lemma \ref{Lem:relaxation_time_by_contraction}, the following bound holds
	\begin{equation}\label{Eq:SpecGap}
		\frac1{\lambda} \le \frac{32 T}{c}\;.
	\end{equation}
	Finally we deduce from Proposition \ref{loc_pi_elliptic} that
	$$ \Gamma P_t f = T |\nabla P_t f|^2 \le M T P_t(|\nabla f|^2) = M P_t \Gamma f\;.$$
	so that the approximate curvature assumption is satisfied. By Theorem \ref{Th:Relaxed} we deduce that for all $\epsilon \in (0,1/2)$
	\begin{align*}
		\wmix(\epsilon) &\le \frac{1}{\lambda \epsilon^3} + \sqrt{\frac{8M\tmix(\epsilon)}{\lambda \epsilon^3}}\\
		&\le \frac{32 T}{c \epsilon^3} + 16\sqrt{2} \sqrt{\frac{T \tmix(\epsilon)}{c\epsilon^3}}\;. 
	\end{align*}
	It remains to prove that the semigroup is indeed a Wasserstein contraction with the aforementioned parameters. To that end, we want to use Theorem \ref{thm_abstrait_monmarche} with $b(x) = -\nabla U(x).$ It is immediate to check that since $\nabla^2 U \geq -K$ we also have $k \geq -K$. Moreover, for any $x,y \in \R^d$, the part of the segment linking $x$ and $y$ within $B(0,R)$ has length at most $2R$.  Hence for any $x$ such that $|x| \geq \Big(1+\sqrt{2}(\frac{2K}{c} +1)\Big)R$ and $y \in \R^d$, we have
\begin{enumerate}
\item Either $|x-y| \leq \sqrt{2}(\frac{2K}{c} +1) R$, in which case the segment linking $x$ and $y$ does not intersect $B(0,R)$, and hence $U$ is uniformly convex along it, so that
$$\langle x-y, \nabla U(x)-\nabla U(y)\rangle \geq c|x-y|^2;$$

\item Either $|x-y| > \sqrt{2}(\frac{2K}{c} +1) R$, in which case 
\begin{align*}
\langle x-y, \nabla U(x)-\nabla U(y)\rangle
&\geq -2KR|x-y| + c|x-y|(|x-y|-R)\\
& = c|x-y|^2-(2K+c)R|x-y|\;.
\end{align*}
\noindent Applying Young's inequality on the second term on the R.H.S., we further get
\begin{align*}
\langle x-y, \nabla U(x)-\nabla U(y)\rangle
&\geq \frac{c}{2}|x-y|^2 - \frac{c}{2}\left(\frac{2K}{c} +1\right)^2R^2\\
&\geq \frac{c}{4}|x-y|^2.
\end{align*}
\end{enumerate}
Hence we get $k(x) \geq c/4$ as soon as $|x| \geq (1+2\sqrt{K/c})R$. So we can apply Theorem \ref{thm_abstrait_monmarche} with $c_0 = c/4$ and $R_0 = \Big(1+\sqrt{2}(\frac{2K}{c} +1)\Big)R$. From the explicit expression of $T_0$, we have the bounds
$$ (2K+c_0)(2K+c_0/2)\frac{R_*^2}{c_0d} \leq T_0 \leq C(K,c)\frac{R^2+\delta R}{d} \;,$$
where $C(K,c)$ is a constant that only depends on $K$ and $c$. From the lower bound we see that for such $T$ we have
$$M \leq 1 + \frac{c_0/2}{2K+c_0/2} < 2\;,$$
as claimed.
\end{proof}

\subsubsection{Cutoff for mean-field diffusions} \label{Sect:elliptic_meanfield}

We now study the application to cutoff for overdamped mean-field dynamics. Wasserstein contraction for these dynamics was established in \cite{Mon23}, and will allow us to apply Theorem \ref{Th:Relaxed}. We will then check that the mixing time goes to infinity to deduce cutoff. We consider the state space $\R^{dn}$ with $d$ the (fixed) dimension of single particles, and $n$ the number of particles, which shall go to infinity. The dynamic is given by
$$dX^i_t = F(X^i_t)dt + \frac{1}{n}\sum_{j=1}^n H(X^i_t - X^j_t)dt + \sqrt{2T}dB^i_t\;,\quad i\in\{1,\ldots,n\}\;.$$

\begin{assumption} \label{assumptions_mean_field}
\begin{enumerate}
    \item There exist $R > 0$ and $C_F, c_F > 0$ such that $$\langle(F(x)-F(y)), (x-y) \rangle \leq \begin{cases} C_F|x-y|^2 & \forall x, y \in \R^d \\
    -c_F|x-y|^2 & \forall x, y \in \R^d \text{ s.t. } |x| \geq R \end{cases}$$
    
    \item $H$ is $\Lambda$-Lipschitz and bounded.
    
    \item $2\Lambda < c_F$
\end{enumerate}
\end{assumption}

We define
\begin{equation}
    R_* := R\left(2 + \frac{2(C_F + 2\beta)}{c_F-2\beta}\right)^{1/(2d)}\;,\quad K_* := C_F + \beta + \frac{c_F}{2}
\end{equation}
and 
\begin{equation}
    T_0 := \frac{K_*}{d(c_F-2\beta)}\left(R_*^2(C_F+2\beta) + 4||H\|_\infty R_* + 2\sup_{|x|\leq R_*} -F(x)\cdot x\right).
\end{equation}

\begin{theorem} \label{thm_cutoff_mean_field}
Assume that the diffusion starts from an initial data of the form $(x_0,\ldots,x_0) \in \R^{dn}$ with $|x_0|$ uniformly bounded in $n$. Under Assumption \ref{assumptions_mean_field}, cutoff occurs as soon as $T \geq T_0$ and, for some $\varepsilon \in (0,1)$, ${\rm t}_{\textsc{mix}}^{(n)}(\varepsilon) \to \infty$.

In particular, cutoff occurs if 
$$T \geq T_0 \text{ and } \Lambda \leq \frac{c_F}{2} \exp\Big(-(c_F + C_F)\frac{R^2}{8T}\Big).$$
\end{theorem}

As in Remark \ref{Rem:initial_data_mf}, the assumption that all particles start form the same position is only used to apply existing results on propagation of chaos for mean-field dynamics. 

\begin{remark}
We do not believe the second assumption on the temperature is truly necessary for cutoff to occur. It arises as a technical condition to apply the results of \cite{HYY24} on propagation of chaos in total variation distance. The extra condition on $T$ can be replaced by the assumption that propagation of chaos in total variation distance holds for the mean-field particle system. 
\end{remark}

We shall apply \cite[Theorem 4.1]{Mon23} to get Wasserstein contraction estimates, following the reasoning in \cite[Section 4.2]{Mon23}. We consider a more restrictive setting than that of \cite{Mon23}, in order to simplify the exposition and the explicit expressions of various constants. 

\begin{lemma} \label{lemma_lip_mean_field}
	For all $x,y \in \R^{dn}$
	$$\frac{1}{n}\sum_{i,j=1}^n |H(x_i-x_j)-H(y_i-y_j)| |x_i-y_i| \leq 2\Lambda \|x-y\|^2\;.$$
\end{lemma}
\begin{proof}
	Since $H$ is $\Lambda$-Lipschitz
	$$ |H(x_i-x_j)-H(y_i-y_j)| \leq \Lambda \big| (x_i-x_j) - (y_i-y_j) \big| \le \Lambda \big(|x_i-y_i| + |x_j-y_j|\big)\;.$$
	Using the bound $(1/n) (\sum_i z_i)^2 \le \sum_i z_i^2$ we obtain
	\begin{align*}
		\frac{1}{n}\sum_{i,j=1}^n |H(x_i-x_j)-H(y_i-y_j)| |x_i-y_i| &\leq \Lambda \sum_i |x_i-y_i|^2 + \frac{\Lambda}{n} \big(\sum_{i} |x_i-y_i|\big)^2\\
		&\leq 2\Lambda \|x-y\|^2\;.
	\end{align*}
\end{proof}
\begin{proof}[Proof of Theorem \ref{thm_cutoff_mean_field}]
	Assumptions \ref{assumptions_mean_field} and Lemma \ref{lemma_lip_mean_field} allow us to apply \cite[Theorem 4.1]{Mon23} with $C_G = a = 2\Lambda$, and get Wasserstein contraction with $\varphi(t) = \sqrt{M}e^{-\varrho t}$ uniformly in $n$, where
$$\varrho = \frac{dT(c_F-2\Lambda)}{4dT + 8K_*R_*^2}\;,\quad M := 1 + \frac{2K_*R_*^2}{dT}.$$
We aim at applying Theorem \ref{Th:Relaxed}, let us check that the required assumptions are satisfied. By Proposition \ref{Prop:Kuwada} we see that approximate curvature holds with constant $M$.  The chain rule is satisfied since we deal with a diffusion process. By Lemma \ref{Lem:relaxation_time_by_contraction}, we deduce the following lower bound on the Poincar\'e constant:
$$ \lambda > \frac{\varrho}{MT}\;.$$
Therefore by Theorem \ref{Th:Relaxed}, cutoff occurs as soon as ${\rm t}_{\textsc{mix}}^{(n)}(\varepsilon)$ goes to $\infty$ as $n\to\infty$.

To prove that the mixing time goes to infinity, we will rely on the extra condition on temperature, and apply known results on propagation of chaos. Let us denote by $\rho^n_t$ the law of the full system at time $t$ starting from $(x_0,\ldots,x_0)$, $\rho^{k,n}_t $ the law of the first $k$ particles, and $\pi^k_n$ their law at equilibrium. 
We also denote $\bar{\rho}_t$ the solution to the McKean-Vlasov PDE that describes the mean-field limit: 
$$\partial_t\bar{\rho}_t = \Delta \bar{\rho}_t + \operatorname{div}(\bar{\rho}_t(F + \bar{\rho}_t \ast H)$$
and $\bar{\pi}$ its invariant probability measure. We have
\begin{align*}\|\rho^n_t -\pi^n\|_{TV} &\geq \|\rho^{k,n}_t - \pi^k_n\|_{TV} \\
&\geq -\|\rho^{k,n}_t - \bar{\rho}_t^{\otimes k}\|_{TV} + \|\bar{\rho}_t^{\otimes k} - \bar{\pi}^{\otimes k}\|_{TV} - \|\bar{\pi}^{\otimes k} - \pi^k_n\|_{TV}
\end{align*}
For any $\theta \in (0, 1/2)$, applying \cite[Theorem 1.1]{HYY24} with $\delta/(1+\delta) = 1/2 - \theta$ (and some elementary computations to check the assumptions), whenever (recall that $\Lambda$ is an upper-bound of the Lipschitz semi-norm of $H$)
$$\Lambda \leq \frac{c_F}{2} \exp\Big(-(c_F + C_F)\frac{R^2}{8T}\Big),$$
we have propagation of chaos in total variation distance, both at time $t \geq 1$ and at equilibrium:  
\begin{align*}
	\|\bar{\pi}^{\otimes k} - \pi^k_n\|_{TV} &\leq C(1+|x_0|)kn^{-1/2+\theta}\;,\\
	\|\rho^{k,n}_t - \bar{\rho}_t^{\otimes k}\|_{TV} &\leq C(1+|x_0|)kn^{-1/2+\theta}\;.
\end{align*}
Hence
$$\|\rho^n_t -\pi^n\|_{TV} \geq \|\bar{\rho}_t^{\otimes k} - \bar{\pi}^{\otimes k}\|_{TV} - CkN^{-1/2 + \theta}.$$
Taking $\theta = 1/8$ and $k = n^{1/4}$, we deduce that for all $t\ge 1$
$$\liminf_n \|\rho^n_{t} -\pi^n\|_{TV} \geq \liminf_k  \|\bar{\rho}_t^{\otimes k} - \bar{\pi}^{\otimes k}\|_{TV}\;.$$
If we show that for any given $t\ge 1$, the r.h.s.~goes to $1$ then we will deduce that ${\rm t}_{\textsc{mix}}^{(n)}(\varepsilon)$ goes to $\infty$ as $n\to\infty$.
Denote by $\He(\nu,\mu)$ the Hellinger distance between two probability measures $\nu$ and $\mu$. Recall that
$$ \He(\nu,\mu)^2 \le \|\nu-\mu\|_{TV} \le \He(\nu,\mu)\;,$$
and
$$ \He(\otimes_{i=1}^k \nu_i,\otimes_{i=1}^k \mu_i)^2 = 1 - \prod_{i=1}^k (1-\He(\nu_i,\mu_i)^2)\;.$$
We thus get
$$ \|\bar{\rho}_t^{\otimes k} - \bar{\pi}^{\otimes k}\|_{TV} \ge \He(\bar{\rho}_t^{\otimes k} , \bar{\pi}^{\otimes k})^2 = 1 - (1-\He(\bar{\rho}_t , \bar{\pi})^2)^k\;.$$
As a consequence, if we take $t = \bar{\rm t}_\textsc{mix}(\delta)$ for some constant $\delta>0$ and where $\bar{\rm t}_\textsc{mix}$ stands for the mixing time of $\bar{\rho}$ then
$$ \liminf_k \|\bar{\rho}_t^{\otimes k} - \bar{\pi}^{\otimes k}\|_{TV} = 1\;.$$
Since $\bar{\rm t}_\textsc{mix}(\delta) \to \infty$ as $\delta \downarrow 0$, we can conclude.
\end{proof}

\subsubsection{An alternative approach : regularity estimates and transport maps}

Our goal in this subsection is to present an alternative approach for proving local Poincar\'e inequalities for diffusion processes, using transport maps and regularity estimates for semigroups. It builds up on a construction of Lipschitz transport maps due to Kim and Milman \cite{KiMi12}, which was revisited in various recent works \cite{MiSh23,FMS24,CoEi25}. It has also been the subject to recent attention due to the connection with diffusion-based generative models \cite{SSKKEP21,CPZ26}. We once again consider a diffusion process of the form
$$dX_t = -\nabla U(X_t)dt + \sqrt{2}dB_t.$$
We also assume that $e^{-U}$ is a probability density. 

Let $p_t(x,\cdot)$ be the density with respect to the equilibrium probability measure $e^{-U}$ of the law of $X_t$ starting from $x$. It satisfies the PDE
$$\partial_t(p_t\pi) = \operatorname{div}(\pi\nabla p_t).$$ 
By the chain rule, this is also
$$\partial_t(p_t\pi) = \operatorname{div}(p_t\pi(\nabla \log p_t)),$$
and hence the flow of marginals of this PDE is also the law of the solution to the ODE with random initial data
$$\dot{y}_s = -\nabla \log p_t(x,y_s); \hspace{3mm} y_0 = x.$$
As was observed by Kim and Milman \cite{KiMi12}, we can therefore construct a transport map from the equilibrium measure $\pi$ onto the law at time $t$ by solving this ODE. Indeed, if $S_{t,s}(y)$ is the semigroup obtained by running the above ODE (assuming it is well-defined and has unique solutions) for a time $t$, and if $S_{t,\infty}$ is invertible, then $S_{t,\infty}^{-1}$ is a transport map from $\pi$ onto $p_t(x,\cdot)\pi$. Moreover, we can use regularity properties of $\log p_t$ to deduce regularity properties of the flow, and its inverse. We shall use the following version, which is an immediate consequence of \cite[Proposition 1]{FMS24} (itself a straightforward adaptation of \cite{KiMi12}). 

\begin{theorem}
Assume that $\log p_t$ is $C^2$ and satisfies the one-sided Hessian bound
$$\nabla^2_y \log p_t(x,y) \leq \beta_t \operatorname{Id}$$
in the sense of symmetric matrices. Then for any $t$ and any $x$, the probability measure $p_t(x,y)d\pi(y)$ is the image of the invariant measure $\pi$ by a $\exp\left(\int_t^\infty{\beta_sds}\right)$-Lipschits map. 
\end{theorem}

In the literature, Lipschitz regularity of such transport maps is typically analyzed for nicer initial data than Dirac masses. However, it is always studied up to time zero, which complicates the analysis, since the short time behavior of $||\nabla^2 \log p_t(x,\cdot)||_{op,\infty}$ might not be integrable for non-smooth initial data. Here we only need to estimate it starting at some positive time, which significantly simplifies the analysis. 

The key application to cutoff is that Lipschitz transport maps allow one to transfer Poincar\'e inequalities: 

\begin{lemma}
Let $T$ be a transport map sending $\mu$ onto $\nu$. Then
$$C_P(\nu) \leq ||T||_{lip}^2C_P(\mu).$$
\end{lemma}
Here $C_P(\mu)$ denotes the sharp constant in the Poincar\'e inequality
$$\Var_{\mu}(f) \leq C\int{|\nabla f|^2d\mu} \hspace{3mm} \forall f \text{ s.t. } |\nabla f| \in L^2(\mu).$$
The proof of this lemma is a completely straightforward change of variable argument, see for example \cite[Corollary 8]{Caf00}. 

\begin{corollary}
Let $\alpha_n := \sup_x \int_{t_{rel}}^{+\infty}{||\nabla^2 \log p^{(n)}_t(x,\cdot)||_{op,\infty}dt}$. If $\alpha_n$ and $t_{rel}^{(n)}$ are finite for all $n$, then the cutoff window satisfies
$$w_{mix}(\epsilon) \leq \frac{3\alpha_n^2 t_{rel}^{(n)}}{\epsilon^2}.$$
In particular, cutoff occurs as soon as 
$$\frac{t_{mix}}{\alpha_n^2 t_{rel}} \longrightarrow +\infty.$$
\end{corollary}

In particular, the product condition implies cutoff for sequences of drift-diffusion SDEs as long as $\alpha_n$ remains bounded.

The proof is once again a trivial adaptation of \cite{Sal25}, just using the fact that for any $t \geq t_{rel}$ a local Poincar\'e inequality holds with constant $\alpha_n^2t_{rel}$ (since $t_{rel}$ is also the Poincar\'e constant for $\pi$, as the inverse of the spectral gap). 

\begin{remark}
This type of result can be extended to diffusion processes with non-constant diffusion matrices, as well as to diffusions on manifolds. The key estimate to be proved is to study lipschitz regularity of the diffusion transport map. See for example \cite{FMS24,Lop25} for some results (though with non-Dirac initial data) on manifolds satisfying certain curvature type assumptions.
\end{remark}

\begin{example}[Ornstein-Uhlenbeck process]
Consider the Ornstein-Uhlenbeck process in $\R^d$. We have $t_{rel} = 1$ and
$$p_t(x,y) = (1-e^{-2t})^{-d/2}\exp\left(\frac{|y|^2}{2}-\frac{|e^{-t}x-y|^2}{2(1-e^{-2t})}\right).$$
Hence 
\begin{align*}
\int_{t_{rel}}^{+\infty}{||\nabla^2 \log p^{(n)}_t(x,\cdot)||_{op,\infty}dt} &= \int_1^{+\infty}{\frac{e^{-2t}}{1-e^{-2t}}dt} \\
&\frac{-\log(1-e^{-2})}{2}.
\end{align*}
and cutoff occurs (since the product condition holds). This recovers the well-known cutoff phenomenon for the Ornstein-Uhlenbeck process \cite{BCL23}. Note that since we actually only need a one-sided bound on $\nabla^2\log p_t$, we could actually get the sharp Lipschitz constant $1$. 
\end{example}

\subsection{Discrete chains and random walks on groups}

\subsubsection{Proof of Theorem \ref{Th:RelaxedDiscrete}}

The proof very closely follows that of \cite{PeSa25}, so we only give a sketch of the parts that are left unchanged. As in Section 3.1, the only difference is to use approximate curvature instead of usual Bakry-Emery curvature to ensure a local Poincar\'e inequality holds. Given a function $g: E\to\R$, we denote by $\Lip(g)$ the best constant $L > 0$ such that
$$T(x,y) > 0 \Rightarrow |g(x)-g(y)| \leq L,$$ 
i.e. the Lipschitz constant of $g$ with respect to the metric induced by the graph structure associated with the transition rates of the Markov chain. Recall that we assumed the symmetry condition \eqref{Eq:Symmetry}, so that this graph is undirected.  For later use, we recall the entropy decay bound (which follows~\eqref{Eq:genericDerivEnt})
$$ \Ent(t) \le e^{-\alpha (t-s)} \Ent(s)\;,\quad \forall \;0 \le s \le t\;,$$
where $\alpha$ is modified log-Sobolev constant defined in~\eqref{Eq:mLSI}.

Fix \(\varepsilon\in(0,1/2)\) and set
\[
    t_0:=\tmix(1-\varepsilon).
\]
Fix an initial state \(x\).  Let \(f_t\) be the density of \(P_t(x,\cdot)\) with respect to \(\pi\), and keep the notation previously introduced
\[
    \Ent(t):=\Ent(f_t),
    \qquad
    \Varent(t):=\Varent(f_t).
\]
All estimates below are uniform over all possible initial states \(x\).

We first recall an inequality derived from two discrete estimates in~\cite[Lemmas 3--4]{PeSa25}. Under the sparsity assumption and the entropy decay bound, for all \(t>0\),
\[
    1+\Lip(\log f_t)
    \le
    15\left(\log d+\frac{1}{\alpha t}\right).
\]
Combined with Pinsker's inequality, we obtain a standard bound for the chain started from this fixed \(x\)
\begin{equation}\label{Eq:entropymixingtime}
    t_{\mathrm{mix}}^x(\varepsilon)
    \le
    t+\alpha^{-1}
    \log\!\left(
        1\vee\frac{\Ent(t)}{2\varepsilon^2}
    \right),
    \qquad t\ge0.
\end{equation}
Following the arguments presented above~\eqref{Eq:LocalPoincaré}, the approximate curvature assumption gives the local Poincar\'e estimate
\[
    \Var\bigl(g(X_t)\mid X_0=x\bigr)
    \le
    2Mt\,\E\bigl[\Gamma g(X_t)\mid X_0=x\bigr].
\]
The approximate chain-rule argument of \cite[Proposition 1]{PeSa25} therefore yields
\[
    \frac{d}{d t}\Ent(t)
    \le
    -\frac{\Varent(t)}
    {2Mt\bigl(1+\Lip(\log f_t)\bigr)}.
\]
Combining this with the preceding Lipschitz estimate, we obtain
\[
    \frac{d}{d t}\Ent(t)
    \le
    -\frac{\Varent(t)}
    {30M\bigl(t\log d+\alpha^{-1}\bigr)}.
\]

For \(t\ge t_0\), the definition of \(t_0\) gives
\[
    \|P_t(x,\cdot)-\pi\|_{\TV}\le 1-\varepsilon.
\]
Whenever \(\Ent(t)>1/\varepsilon\), the reverse Pinsker inequality~\eqref{Eq:ReversePinsker} implies
\[
    \Varent(t)\ge \bigl(\varepsilon \Ent(t)-1\bigr)^2.
\]
Consequently, whenever \(t\ge t_0\) and \(\Ent(t)>1/\varepsilon\),
\[
    \frac{d}{d t}\Ent(t)
    \le
    -\frac{\bigl(\varepsilon \Ent(t)-1\bigr)^2}
    {30M\bigl(t\log d+\alpha^{-1}\bigr)}.
\]
Integrating this differential inequality between \(t_0\) and \(t>t_0\), and using that \(s\mapsto (s\log d+\alpha^{-1})^{-1}\) is decreasing, gives
\[
    \Ent(t)
    \le
    \frac1\varepsilon
    +
    \frac{30M\bigl(t\log d+\alpha^{-1}\bigr)}
    {\varepsilon^2(t-t_0)}.
\]
The same upper bound is trivially valid if \(\Ent(t)\le1/\varepsilon\).

Plugging this entropy estimate into the entropy mixing-time bound, we obtain, for all \(t>t_0\),
\[
\begin{aligned}
    t_{\mathrm{mix}}^x(\varepsilon)
    &\le
    t
    +\alpha^{-1}\log\!\left(\frac{30M}{\varepsilon^4}\right)+\alpha^{-1}
    \log\!\left(
        1+\frac{t\log d+\alpha^{-1}}{t-t_0}
    \right),
\end{aligned}
\]
Choosing \(t=t_0+\alpha^{-1}\), using the bound \(\alpha^{-1} \leq \alpha^{-1} \log 3\), and then taking the supremum over the initial state \(x\), yields
\[
\begin{aligned}
	\tmix(\varepsilon)-\tmix(1-\varepsilon)
	&\le
	\alpha^{-1}\log\!\left(\frac{90M}{\varepsilon^4}\right)\\
	&\quad
	+\alpha^{-1}\log(1+\log d)
	+\alpha^{-1}\log\!\left(2+\alpha t_0\right).
\end{aligned}
\]
Since \(t_0=\tmix(1-\varepsilon)\le \tmix(\varepsilon)\), this proves the announced window estimate.

Finally, suppose that the chains depend on \(n\), and that for some \(\delta\in(0,1)\),
\[
    \frac{\alpha_n{\rm t}_{\textsc{mix}}^{(n)}(\delta)}
    {1+\log M_n+\log(1+\log d_n)}
    \longrightarrow+\infty.
\]
Choosing \(\varepsilon<1-\delta\) gives
\[
    {\rm t}_{\textsc{mix}}^{(n)}(1-\varepsilon)\le {\rm t}_{\textsc{mix}}^{(n)}(\delta).
\]
The window estimate above, together with
\[
    \log\bigl(2+\alpha_n{\rm t}_{\textsc{mix}}^{(n)}(\delta)\bigr)
    =
    o\bigl(\alpha_n{\rm t}_{\textsc{mix}}^{(n)}(\delta)\bigr),
\]
implies
\[
    \frac{{\rm w}_{\textsc{mix}}^{(n)}(\varepsilon)}
    {{\rm t}_{\textsc{mix}}^{(n)}(\delta)}
    \longrightarrow0.
\]
Since this holds for arbitrarily small \(\varepsilon\in(0,1/2)\), the usual monotonicity argument yields cutoff.

\subsubsection{Proof of Corollary \ref{Cor:rw_symmetric_group}}

We deal with the continuous-time walk on ${\mathfrak S}_n$ with generator
\[
\mathcal{L}_n f(\sigma)
=
\sum_{\tau\in\cT_n}\nu_n(\tau)\bigl(f(\tau\sigma)-f(\sigma)\bigr),
\]
where \(\cT_n\) denotes the set of transpositions and \(\nu_n\) is a probability measure on \(\cT_n\). We aim at applying Theorem \ref{Th:RelaxedDiscrete}, we thus need to check that the required assumptions hold. The sparsity assumption is satisfied with $d_n = |\cT_n|=n(n-1)/2$. We turn to the entropy decay. Since the invariant measure of the process (which we denote by $\pi_n$) is uniform, from the classical modified logarithmic Sobolev inequality (mLSI) for uniform random transpositions \cite{BoTe06,GaQu03,Goe04} we have
$$\Ent(f) \leq \frac{n-1}{2}\sum_{s \in {\mathfrak S}_n, \tau \in \cT_n} \frac{2}{n(n-1)}(f(\tau s)-f(s))(\log f(\tau s) - \log f(s))\pi_N(s).$$
Therefore, since
$$\frac{2}{n(n-1)} \leq \sup_\tau \nu_n(\tau)\;, \hspace{3mm} \inf_\tau \nu_n(\tau) \leq \nu_n\;,$$
we have
\begin{align*}
	\Ent(f) &\leq \frac{n-1}{2}\frac{\sup \nu_n(\tau)}{\inf \nu_n(\tau)}\sum_{s \in {\mathfrak S}_n, \tau \in \cT_n}(f(\tau s)-f(s))(\log f(\tau s) - \log f(s))\nu_n(\tau)\pi_N(s)\\
	&\leq \frac{n-1}{2}\frac{\sup \nu_n(\tau)}{\inf \nu_n(\tau)} \cE_n(f,\log f)\;,
\end{align*}
where $\cE_n$ is the Dirichlet form of the process. From this inequality, we deduce that the modified log-Sobolev constant defined in~\eqref{Eq:mLSI} satisfies
$$\alpha_n^{-1} \leq \frac{n-1}{2}\frac{\sup \nu_n(\tau)}{\inf \nu_n(\tau)}.$$

Let us now check the approximate curvature assumption. Denote by $\nu_n^k$ the law of the product of $k$ transpositions chosen independently, each according to $\nu_n$. Setting
$$ M := \sup_{n\ge 1} \frac{\max_{\tau\in {\mathcal T}_n}\nu_n(\tau)}{\min_{\tau\in{\mathcal T}_n}\nu_n(\tau)}\;,$$
which is finite by assumption, we have
\begin{align*}
P_t(\Gamma f)(s) &= \sum_{k \in \N} \frac{t^ke^{-t}}{k!}\sum_{p \in \cT_n}\sum_{\tau \in \cT_n} (f(\tau p s) - f(ps))^2\nu_n^{ k}(p)\nu(\tau) \\
&\overset{\sigma = p^{-1}\tau p}{=} \sum_{k \in \N} \frac{t^ke^{-t}}{k!}\sum_{p \in \cT_n}\sum_{\sigma \in \cT_n} (f(p\sigma s) - f(ps))^2\nu_n^{ k}(p)\nu(p\sigma p^{-1}) \\
&\geq M^{-1} \sum_{k \in \N} \frac{t^ke^{-t}}{k!}\sum_{p \in \cT_n}\sum_{\sigma \in \cT_n} (f(p\sigma s) - f(ps))^2\nu_n^{ k}(p)\nu(\sigma) \\
&\geq M^{-1}\sum_{\sigma \in \cT_n} \left( \sum_{k \in \N} \frac{t^ke^{-t}}{k!}\sum_{p \in \cT_n} f(p\sigma s) - f(ps) \nu_n^{ k}(p)\right)^2\nu(\sigma) \\
&= M^{-1}\Gamma(P_tf)(s)
\end{align*}
which is indeed the approximate curvature condition. To go from the third to the fourth line, we used Jensen's inequality. 

All assumptions of Theorem \ref{Th:RelaxedDiscrete} are satisfied, we thus deduce from~\eqref{Eq:CutoffDiscrete}, the value of $d_n$ and the lower bound on $\alpha_n$ that cutoff occurs as soon as
$$\frac{{\rm t}_{\textsc{mix}}^{(n)}(\varepsilon)}{n \log \log n} \longrightarrow +\infty\;.$$
To establish such a lower bound on the mixing time,  we rely on a coupon-collector argument applied to the number of labels that have not yet been touched.  In the uniform case, this idea is classical in the analysis of random transpositions by Diaconis--Shahshahani \cite[Theorem 2]{DiSh81}; see also the fixed-point lower bound in \cite[Section 8.2.1, Proposition 8.4]{LPW}.  In the biased case, Nestoridi--Yan \cite[Theorems 1.1 and 1.3]{NeYa24} prove a much sharper result, including cutoff and the Poisson asymptotics for the number of fixed cards near the critical time.  For our criterion, however, only a rough \(n\log n\) lower bound is needed.

\begin{proposition}[Coupon-collector lower bound]
Under the preceding assumption, for every \(\varepsilon\in(0,1)\), there exists \(c_M>0\) such that
\[
    {\rm t}_{\textsc{mix}}^{(n)}(\varepsilon) \ge c_M n\log n
\]
for all \(n\) large enough.  In particular,
\[
    \frac{{\rm t}_{\textsc{mix}}^{(n)}(\varepsilon)}{n\log\log n}\longrightarrow+\infty.
\]
\end{proposition}

\begin{proof}
By left-translation invariance of the walk and because the stationary measure is uniform on ${\mathfrak S}_n$, it is enough to start from the identity. For \(i\in\{1,\ldots,n\}\), let
\[
    r_i:=\sum_{j\ne i}\nu_n((ij))
\]
be the rate at which label \(i\) is touched.  The bounded-bias assumption implies
\[
    \max_{\tau\in\cT_n}\nu_n(\tau)
    \le \frac{M}{|\cT_n|}
    =
    \frac{2M}{n(n-1)},
\]
and therefore
\[
    r_i\le (n-1)\max_{\tau\in\cT_n}\nu_n(\tau)\le \frac{2M}{n}.
\]
Let \(U_t\) be the number of labels that have not been touched before time \(t\).  Then
\[
    \E (U_t)
    =
    \sum_{i=1}^n e^{-r_i t}
    \ge
    n\exp\!\left(-\frac{2Mt}{n}\right).
\]
Take
\[
    t_n=\frac{1-\eta}{2M}\,n\log n,
    \qquad \eta\in(0,1).
\]
Then \(\E (U_{t_n})\ge n^\eta\to+\infty\).

It remains to check that \(U_{t_n}\) does not fluctuate too much.  If \(A_i\) denotes the event that label \(i\) has not been touched before time \(t\), then, for \(i\ne j\),
\[
    \P(A_i\cap A_j)
    =
    \exp\{-(r_i+r_j-\nu_n((ij)))t\}
    \le
    e^{t\max_\tau\nu_n(\tau)}
    \P(A_i)\P(A_j).
\]
Since \(t_n\max_\tau\nu_n(\tau)=O((\log n)/n)=o(1)\), it follows that
\[
    \E (U_{t_n}^2)
    \le
    \E (U_{t_n})
    +
    (1+o(1))(\E (U_{t_n}))^2.
\]
Thus \(U_{t_n}/\E (U_{t_n})\to1\) in probability.  In particular, if \(a_n=(\E( U_{t_n}))^{1/2}\), then
\[
    \P(U_{t_n}\ge a_n)\longrightarrow1.
\]
Write
\[
    \Fix(\sigma):=\{i\in\{1,\ldots,n\}:\sigma(i)=i\}
\]
for the set of fixed points of a permutation \(\sigma\).  Every untouched label belongs to \(\Fix(X_t)\).  On the other hand, under the uniform measure \(\pi_n\) on \(\mathfrak S_n\), the number \(F_n:=|\Fix(\sigma)|\) of fixed points satisfies \(\E_{\pi_n}F_n=1\).  By Markov's inequality,  
\[
    \pi_n(F_n\ge a_n)\le \frac1{a_n}\longrightarrow0.
\]
For the set \(B_n=\{\sigma:|\Fix(\sigma)|\ge a_n\}\), we obtain   \( \P(X_{t_n}\in B_n)\to1\) and \( \pi_n(B_n)\to0\).
Consequently,
\[
    \bigl\|\mathcal L(X_{t_n})-\pi_n\bigr\|_{\TV}\to1,
\]
which proves the claimed lower bound.
\end{proof}

\subsection{Hypoelliptic diffusions}

\subsubsection{Main abstract statement}

Our goal is to establish a cutoff criterion for hypoelliptic diffusions in $\R^d \times \R^d$, in the limit where $d$ becomes large. Note that for such dynamics the carr\'e du champ operator $\Gamma$ differs from $|\nabla|^2$, in particular for kinetic diffusions $\Gamma=|\nabla_v|^2$ which is smaller than $|\nabla|^2 = |\nabla_x|^2 + |\nabla_v|^2$. In the sequel, we shall use $q$ for the full dimension of the space in which the Markov process evolves, so that for kinetic diffusions we have $q = 2d$.\\
We establish a general result cutoff result for a sequence of dynamics in $\R^{q_n}$ satisfying $\Gamma \le |\nabla|^2$. Let us introduce the following notations: given a probability measure $\mu$ admitting a density $g$ w.r.t.~$\pi$, we set
$$ I_\pi(\mu) := \int \frac{|\nabla g|^2}{g} d\pi\;,\quad \Ent_\pi(\mu) := \int g \log g \,d\pi\;.$$
In the particular case where $\mu = f_t d\pi$, we have $I_\pi(\mu) = I(t)$ and $\Ent_\pi(\mu) = \Ent(t)$.

\begin{theorem} \label{abstract_thm_cutoff_hypoelliptic}
	Consider a process in $\R^q$ such that $\Gamma \le |\nabla|^2$ and such that the following three assumptions:
	\begin{enumerate}
		\item Short-time entropic regularization: there exists $\varphi_1 : \R_+^* \longrightarrow \R_+$ such that for all probability measure $\mu$ and all $s,t\ge 0$
		$$ \Ent_\pi(P_{t+s}^* \mu) \leq \varphi_1(t) W_2(P_s^*\mu, \pi)^2\;.$$
		\item Short-time Fisher information regularization: there exists $\varphi_2 : \R_+^* \longrightarrow \R_+$ such that for all probability measure $\mu$ and all $s,t\ge 0$
		$$I_\pi(P_{t+s}^* \mu) \leq \varphi_2(t)\Ent_\pi(P_{s}^* \mu)\;.$$
		\item Wasserstein hypocoercivity: there exist $\varrho, M > 0$ such that for all probability measures $\mu,\nu$ and all $t\ge 0$
		$$W_2(P^*_t\mu, P^*_t\nu)^2 \leq Me^{-2 \varrho t}W_2(\mu, \nu)^2\;.$$
	\end{enumerate}
	Assume that a sequence of process $(X^{(n)})_{n\ge 1}$ in $\R^{q_n}$ satisfies the above assumptions for fixed $\varphi_2$, $M$ and $\varrho$, and for $n$-dependent $\varphi_1^{(n)}$. Assume furthermore that each $X^{(n)}$ starts from some deterministic initial condition in a $W_2$ ball of radius $R_n$ around the invariant measure $\pi_n$. Then cutoff occurs as soon as there exists $\varepsilon \in (0,1)$ such that
	$${\rm t}_{\textsc{mix}}^{(n)}(\varepsilon) \gg 1+ \log \log (\varphi_1^{(n)}(1)R^2_n).$$
\end{theorem}

Fisher information regularization can be obtained using \cite[Theorem 9]{Mon19} or \cite[Theorem A15]{Vil09}. Entropic regularization is also discussed in \cite{Vil09}, with non-explicit constants. A more precise result appears in \cite{GuWa12}. It is not necessary that it is Wasserstein distance that appears in the upper bound for the regularization, all we need is an upper bound $\alpha_n$ on entropy after some time that is uniform in the initial data considered, and that the mixing time is much larger than $\log \log \alpha_n$. The motivation for stating Assumption 1 in this form is to connect to the literature on reverse transport-entropy inequalities along Markov processes and diffusion PDEs. 

Regularization estimates comparing entropy and Wasserstein distance were used in \cite{ChFa25,PeSa25b} for proving cutoff in the context of elliptic dynamics. 

\begin{remark}
	For a reversible elliptic diffusion process $dX_t = -\nabla U(X_t)dt + \sqrt{2}dB_t$, both short-time entropic and Fisher information regularization hold as soon as $V$ is semi-convex \cite{BGL01}. Hence Theorem \ref{abstract_thm_cutoff_hypoelliptic} could be used to prove cutoff for reversible diffusion processes. However, the mixing-time window bound we obtain would be sub-optimal. 
\end{remark}

The proof follows the varentropy approach to cutoff, and in particular shall use the same reverse Pinsker inequality \cite[Lemma 1]{Sal25} and bound on mixing time by entropy \cite[Lemma 2]{Sal25}. However, compared to \cite{Sal25} we lose two properties : the non-negative curvature bound and the differential relation between entropy and the full Fisher information. As in Section \ref{Sec:Elliptic}, we use Wasserstein contraction estimates (the third assumption of our theorem) to recover a local Poincar\'e inequality to bypass the first issue. The first two assumptions are regularization estimates that will allow us to overcome the second issue. 

\begin{proof}[Proof of Theorem \ref{abstract_thm_cutoff_hypoelliptic}]
	Set $t_0 := t_{mix}(1-\epsilon)$. A classical computation~\cite[Eq. (3.1.20) and (3.1.21)]{BGL14} yields for any regular enough function $g$
	\begin{eqnarray*} 
		\Var[g(X_t)] & = & 2\int_0^t (P_{t-s}\Gamma P_{s}g)(x){\,\rm d}s,
	\end{eqnarray*}
	Since $\Gamma \le |\nabla|^2$ we deduce that
	\begin{eqnarray*} 
		\Var[g(X_t)] & \leq & 2\int_0^t (P_{t-s}|\nabla P_{s}g|^2)(x){\,\rm d}s.
	\end{eqnarray*}
	The third assumption combined with Proposition \ref{Prop:Kuwada} ensures that
	$$|\nabla P_s f|^2 \leq \varphi(s)^2 P_s(|\nabla f|^2)\;,\quad \varphi(s)^2 = Me^{-2 \varrho s}\;,$$
	so that
	\begin{eqnarray*} 
		\Var[g(X_t)] & \leq & 2\int_0^t \varphi(s)^2(P_{t-s}P_{s} |\nabla g|^2)(x){\,\rm d}s\\
		&\leq & \psi(t) P_t |\nabla g|^2(x)\;,
	\end{eqnarray*}	
	where
	$$\psi(t) = 2 \int_0^t M e^{-2 \varrho s} ds \le \frac{M}{\varrho}\;.$$
	Taking $g= \log f_t$ and recalling the definition~\eqref{Eq:I} of $I(t)$ we find
	$$ \Varent(t) \le \psi(t) I(t)\;.$$
	For $t \geq t_0$, using the reverse Pinsker inequality~\eqref{Eq:ReversePinsker} we then get
	$$ \Ent(t) \leq \frac{1 + \sqrt{M \varrho^{-1} I (t)}}{\epsilon}\;.$$
	Hence as long as $\Ent(t) \geq 4/\epsilon$, we have
	$$\Ent(t)^2 \leq \frac{2M I(t)}{\varrho\epsilon^2}.$$
	Using the assumption of short-time Fisher information regularization, we have for any fixed $s>0$
	$$\Ent(t+s)^2 \leq \frac{2M\varphi_2(s)}{\varrho\epsilon^2}\Ent(t),$$
	or equivalently
	$$ \log \Ent(t+s) \leq \frac12 \log \Ent(t) + \frac12 \log \frac{2M\varphi_2(s)}{\varrho\epsilon^2}\;.$$
	Set $b = \max(0, \log \frac{2M\varphi_2(s)}{\varrho\epsilon^2})$ and $u_k := \log(\Ent(t + ks))$ for any integer $k\ge 0$. Using the monotonicity of $(u_k)_k$ we get
	$$u_k \geq \log(4/\epsilon) \Rightarrow u_k \leq \frac{1}{2}u_{k-1} + \frac{b}{2}.$$
	Hence as long as $u_k$ is greater than $\log(4/\epsilon)$, it is dominated by the solution of the recursive equation $x_{k} = (x_{k-1}+b)/2$ with $x_0 = u_0$. Hence
	$$u_k \leq \max\left(\log\left(\frac{4}{\epsilon}\right), 2^{-k}u_0 + b\right).$$
	Setting $k_0 := \lceil \log_2(u_0)\rceil$, we see that at time $t+k_0s$ we have
	$$\Ent(t+k_0s) \leq \max\left(\frac{4}{\epsilon}, \exp(1 + b)\right).$$
	This estimate depends on $\Ent(t)$ through $k_0$, which is of order $\log \log \Ent(t)$. 
	
	Now take $t:=t_0 + s$, we have
	$$\Ent(t_0+s) \leq \Ent(s) \leq \varphi_1^{(n)}(s)R_n^2\;,$$
	so that we get
	$$\Ent(t_0 + (k_0+1)s) \leq \max\left(\frac{4}{\epsilon}, \exp(1 + b)\right).$$
	Inserting this estimate into~\eqref{Eq:MixEnt}, and using Lemma \ref{Lem:relaxation_time_by_contraction}, we get
	\begin{align*}
		{\rm t}_{\textsc{mix}}^{(n)}(\epsilon) & \leq {\rm t}_{\textsc{mix}}^{(n)}(1-\epsilon) + (1+k_0)s + \frac{1}{\lambda \epsilon}(1+ \max(4\epsilon^{-1}, \exp(1+b))) \\
		&\leq {\rm t}_{\textsc{mix}}^{(n)}(1-\epsilon) + s\big(2+ \log\log(\varphi_1^{(n)}(s)R_n^2)\big) + \frac{M}{\varrho\epsilon}\Big(8\epsilon^{-1}+ \frac{6M\varphi_2(s)}{\varrho\epsilon^2}\Big)\;.
	\end{align*}
	This estimate on the cutoff window implies cutoff in the regime we consider.
\end{proof}

\subsubsection{Application to kinetic diffusion processes}

\begin{proof}[Proof of Theorem \ref{thm_kinetic_cutoff_general}]
We will check that the assumptions of Theorem \ref{abstract_thm_cutoff_hypoelliptic} are fulfilled with good enough constants. 

Assumption 1 (entropic regularization) is ensured by Lemma \ref{lem_reg_entropy_kinetic} below with $\varphi_1(s)$ of the form $CM^2d_n \max(s, s^{-3})$ where $C$ is a numerical constant that does not depend on any parameter. 

Assumption 2 (Fisher information regularization) is a consequence of \cite[Theorem A.18]{Vil09}. In the notations of the cited result, we can check its assumptions are satisfied with $N_c = 1$ and
$$C_0 = A = \nabla_v; \hspace{3mm} B = v\cdot \nabla_x -\nabla H(x) \cdot \nabla_v, \hspace{3mm} C_1 = \nabla_x;$$
$$Z_1 = 1\;,\hspace{3mm} Z_2 = \Hess H; \hspace{3mm} R_1 = 0.$$
Our assumptions on $H$ ensure $Z_2$ is bounded from above and below by positive constants that do not depend on $d_n$. Therefore the conclusion of \cite[Theorem A.15]{Vil09} is Fisher information regularization with $\varphi_2(t) = C/t^3$ and $C$ a numerical constant that only depends on $m$ and $\Lambda$.

Assumption 3 (Wasserstein contractivity) is ensured by \cite[Theorem 2.12]{Bau17}, with constants $(M, \lambda)$ that are uniform in $n$ (the uniformity is not stated there, but can be checked from the proof). 
\end{proof}

The following regularization estimate for entropy, which we used in the above proof, is a direct consequence of \cite[Corollary 4.7 part (1)]{GuWa12}: 
\begin{lemma} \label{lem_reg_entropy_kinetic}
Consider a dynamic of the form \eqref{eq__kin_dynamic} on $\R^d$, and assume that $\nabla H$ is $\Lambda$-lipschitz. Then 
$$\Ent_{\pi}(P_t^*\mu) \leq \varphi_1(t)W_2(\mu, \pi)^2$$
with $\varphi_1(t) = 2t \max((\Lambda d^{1/2} + 3d^{1/2}/(2t) + 6/t^2)^2, (d^{1/2} + 4/t + 4\Lambda d^{1/2}t/27)^2)$. 
\end{lemma}
\begin{remark}It is possible to get an estimate for non-Lipschitz $\nabla H$ via \cite[Corollary 4.7 part (2)]{GuWa12}. The effective dependence in $d$ may get worse (through higher moments of $\mu$), but is still polynomial in $d$ under some mild assumptions, which would be good enough for our purpose. 
\end{remark}

\begin{remark} \label{rmk_hypo_monmarche}
Our method can also be applied to kinetic diffusions with non-convex potential. As for Corollary \ref{Cor:elliptic_high_temp}, we would have to introduce a temperature parameter that is high enough. See \cite[Section 3.2]{Mon26} for the necessary Wasserstein contraction estimates. 
\end{remark}

\subsubsection{Cutoff for kinetic mean-field dynamics}
\begin{proof}[Proof of Corollary \ref{Cor:Cutoff_kinetic_mf}]
We check the assumptions of Theorem \ref{thm_kinetic_cutoff_general}. 
For \(z=(z_1,\ldots,z_n)\in(\R^p)^n\), one has
\[
\begin{aligned}
    z\cdot \nabla^2 U_n(x)z
    &=
    \sum_{i=1}^n
    z_i\cdot \nabla^2 \hat{U}(x_i)z_i \\
    &\quad+
    \frac1{2n}\sum_{i,j=1}^n
    (z_i-z_j)\cdot
    \nabla^2 W(x_i-x_j)(z_i-z_j).
\end{aligned}
\]
The convexity of \(W\) gives the lower bound
\[
    z\cdot \nabla^2 U_n(x)z\ge m_{\hat{U}}\sum_i |z_i|^2.
\]
For the upper bound, we use
\[
    \frac1{2n}\sum_{i,j=1}^n |z_i-z_j|^2
    =
    \sum_{i=1}^n |z_i-\bar z|^2
    \le
    \sum_{i=1}^n |z_i|^2,
    \qquad
    \bar z=\frac1n\sum_i z_i.
\]
Thus
\[
    \nabla^2 U_n\le (M_{\hat{U}}+L_W)I_{pn}.
\]
Hence the Hessian bounds required in Theorem \ref{thm_kinetic_cutoff_general} hold. Consequently Theorem \ref{thm_kinetic_cutoff_general} shows that cutoff holds as soon as 
\[
{\rm t}_{\textsc{mix}}^{(n)} \gg \log\log(d_n \vee R_n),
\qquad d_n=2pn,
\]
where $R_n$ is the distance between the initial condition and the center of mass, thus completing the proof.
\end{proof}

\begin{remark}[Non-convex or weakly non-convex interactions]
The convexity of \(W\) is not conceptually indispensable.  If
\[
    -\ell_W I_p\le \nabla^2 W\le L_W I_p
\]
and if \(m_{\hat{U}}-\ell_W>0\), the same computation gives
\[
    (m_{\hat{U}}-\ell_W)I_{pn}
    \le
    \nabla^2 U_n
    \le
    (M_{\hat{U}}+L_W)I_{pn}.
\]
Hence the preceding criterion still applies as soon as
\[
    \sqrt{M_{\hat{U}}+L_W}-\sqrt{m_{\hat{U}}-\ell_W}<1.
\]
For non-convex potentials outside this pinched regime, one has to consider the recent work of Monmarch\'e \cite[Theorem 9 and Section 3.2]{Mon26}, in the same spirit as the coupling methods already used for elliptic diffusions \cite[Theorem 1 and Section 4.2]{Mon23}. In fact any result giving a contraction
\[
    W_2(P_t^*\mu,P_t^*\nu)
    \le C e^{-\lambda t}W_2(\mu,\nu),
\]
with \(C,\lambda\) uniform in \(n\), is sufficient.
\end{remark}

Before proving Corollary \ref{Cor:Cutoff_kinetic_mf_application}, we collect a lower bound on the TV distance to equilibrium of the limiting dynamics. Let $(\bar{X},\bar{V})$ be the $\R^p\times \R^p$ valued-process solution of
\begin{equation}\label{Eq:McKeanVlasov}
	\begin{cases}
		d\bar{X}_t &= \bar{V}_tdt\;,\\
		d\bar{V}_t &= -\nabla \hat{U}(\bar{X}_t)dt - \nabla W * \bar{\mu}_t(\bar{X}_t) dt -  \bar{V}_tdt + \sqrt{2}dB_t\;,
	\end{cases}
\end{equation}
where $\bar{\mu}_t$ is the law of $(\bar{X}_t,\bar{V}_t)$.
\begin{lemma}
	Assume that $\hat{U}(x) = (M/2) |x|^2$ for some constant $M>0$, and that $W$ is even and satisfies $0\le \nabla^2 W\le L_W I_p$ with $0 \le L_W < M / ( 4 M + 2)$. Pick a non-zero, deterministic initial condition $(X_0,V_0)$. Then there exist two constants $c,K>0$ such that for all $t>0$
	\begin{equation}\label{Eq:ToProveClaim}
		d_{TV}(\bar{\mu}_{t}, \bar{\mu}_\infty) \ge K e^{-ct}\;.
	\end{equation}
\end{lemma}
\begin{proof}
	Set $f(x,v) = (\alpha x + v) \cdot y$ for some vector $y\in \R^p$ and some parameter $\alpha \in \R$ to be fixed later on. Since $\hat{U}$ and $W$ are even, we deduce that the invariant measure associated to the dynamics $(\bar{X},\bar{V})$ is centered: in particular its projection through the map $f$ has zero mean.\\
	We are going to show that for some choice of parameters $y,\alpha$ there exists $c_0>0$ such that
	\begin{equation}\label{Eq:expectf}
		\E[f(\bar{X}_t,\bar{V}_t)] \ge e^{-c_0 t}\;,\quad t\ge 0\;,
	\end{equation}
	and
	\begin{equation}\label{Eq:varf}
		\sup_{t\ge 0} \Var(f(\bar{X}_t,\bar{V}_t)) < \infty\;,\quad  \Var(f(\bar{X}_\infty,\bar{V}_\infty)) < \infty\;.
	\end{equation}
	With these bounds at hand, by~\cite[Prop 7.12]{LPW} (which also applies in the continuum) we deduce that for some constant $K>0$
	$$ d_{TV}(\bar{\mu}_{t}, \bar{\mu}_\infty) \ge K e^{-2c_0t}\;,\quad t\ge 0\;,$$
	which is the desired result.\\
	We compute
	$$ d\E[f(\bar{X}_t,\bar{V}_t)] = (\alpha - 1) \E[\bar{V}_t\cdot y] - \E[\nabla \hat{U}(\bar{X}_t) \cdot y] - \E[\nabla W* \mu_t(\bar{X}_t)\cdot y]\;.$$
	If we let $\bar{Y}$ be an independent copy of $\bar{X}$ then, using that $W$ is even:
	$$ \E[\nabla W* \mu_t(\bar{X}_t)\cdot y] = \E\big[\E_{\bar{Y}}[\nabla W(\bar{X}_t - \bar{Y}_t)\cdot y]\big] = 0\;.$$
	We write $\hat{U}(x) = (M/2)|x|^2$ and thus find
	$$ \nabla \hat{U}(\bar{X}_t)\cdot y = M \bar{X}_t\cdot y\;.$$
	We now pick $\alpha < 0$ in such a way that $M = \alpha(\alpha-1)$. We find
	\begin{align*}
		d\E[f(\bar{X}_t,\bar{V}_t)] &= -(1-\alpha) \E[\bar{V}_t\cdot y] - M \E[\bar{X}_t \cdot y]\\
		&= -(1-\alpha) \E[f(\bar{X}_t,\bar{V}_t)]\;.
	\end{align*}
	Now pick $y\in\R^p$ in such a way that $f(\bar{X}_0,\bar{V}_0) =1$. This yields~\eqref{Eq:expectf}. Regarding~\eqref{Eq:varf}, this follows from a direct estimation using an appropriate Lyapunov function as in~\cite[Lemma 2.1]{GLBM}, provided $L_W < M / (4M+2)$.
\end{proof}

\begin{proof}[Proof of Corollary \ref{Cor:Cutoff_kinetic_mf_application}]
With the result of Corollary \ref{Cor:Cutoff_kinetic_mf} at hand, we only need to show that $R_n$ grows at most polynomially with $n$ and that
\begin{equation}\label{Eq:lowerbdTmix}
	{\rm t}_{\textsc{mix}}^{(n)} \gg \log \log n.
\end{equation}
Recall that $R_n$ is the distance between the initial condition $(x_{i,n},v_{i,n})_{i=1,\ldots,p} \in \R^{pn}\times \R^{pn}$, and the center of mass $(x^*_n,v^*_n) \in \R^{pn}\times \R^{pn}$ defined by
$$ x^*_{i,n} =  \int x_i \pi_n(dx,dv)\;,\quad v^*_{i,n} =  \int v_i \pi_n(dx,dv)\;,\quad i\in\{1,\ldots,n\}\;.$$
Recall that $\pi_n(dx,dv)=e^{-U_n(x)} dx \otimes \gamma(v)dv$ where $\gamma$ is the $\mathcal{N}(0,1)$ density. By assumption, the Euclidean norm of the initial condition is of order $n$. To control the size of $R_n$, it suffices to control the Euclidean norm of the center of mass. By passing to the limit $t\to\infty$ in~\cite[Theorem 2.3, Eq.(2.22)]{CLRW24}, we deduce that $W_2(\pi_n,\bar{\mu}_\infty^{\otimes n})$ remains bounded uniformly over all $n$ large enough. We have
\begin{align*}
	\sum_{i=1}^n (|x^*_{i,n}|^2 + |v^*_{i,n}|^2) &\le \sum_{i=1}^n \int \big( |x_{i}|^2 + |v_{i}|^2\big) \pi_n(dx,dv)\\
	&\le W_2(\pi_n,\bar{\mu}_\infty^{\otimes n})^2 + n \int (|x|^2 + |v|^2) \bar{\mu}_\infty(dx,dv)\;,
\end{align*}
so that the Euclidean norm of the center of mass is of order at most $\sqrt{n}$, and therefore $R_n$ grows not faster than $n$.

It remains to show~\eqref{Eq:lowerbdTmix}. As in Section \ref{Sect:elliptic_meanfield}, we shall use (entropic) propagation of chaos for the mean-field dynamics, recently proved in \cite{CLRW24}, to get a lower bound on the mixing time. 

Denote by $\mu^n_t$ the law of the system, and $\mu^{n,k}_t$ the law of the $k$ first particles. Under our strong assumptions on the potential and the initial condition, \cite[Theorem 2.3]{CLRW24} applies and yields (for $t$ not too small)
$$\Ent(\mu^N_t|\bar{\mu}_t^{\otimes N}) \leq C_1e^{C_1t}(1+r^2)$$
where $C_1$ is a constant that does not depend on $n$. We used here the fact that we start all particles from the same initial data, which is also the initial data for $\bar{\mu}_t$. The upper bound stated in \cite[Theorem 2.3]{CLRW24} depends on the variance of $\bar{\mu}_t$, but this can be controlled uniformly in time by $C(R^2 +1)$ by using convergence to equilibrium for $\bar{\mu}_t$ in Wasserstein distance, which is itself a consequence of the uniform-in-$n$ contractive estimate for the $n$-particle dynamic. 

From subadditivity of entropy, we deduce that
$$\Ent(\mu^{n,k_n}_t|\bar{\mu}_t^{\otimes k_n}) \leq \left\lfloor \frac{n}{k_n} \right\rfloor^{-1}C_1e^{C_1t}(1+r^2).$$
Taking $t_n = \sqrt{\log n}$ and $k_n = n^{1/2}$, we get a bound of the form
$$\Ent(\mu^{n,k_n}_t|\bar{\mu}_t^{\otimes k_n}) \leq Cn^{-1/2}e^{C\sqrt{\log n}}(1+r^2) \longrightarrow 0.$$
Therefore, from Pinsker's inequality, we have
$$d_{TV}(\mu^{n,k_n}_{t_n}, \bar{\mu}_{t_n}^{\otimes k_n}) \longrightarrow 0.$$
Let us assume for now that the following holds
\begin{equation}\label{Eq:Claim}
	d_{TV}(\bar{\mu}_{t_n}^{\otimes k_n}, \bar{\mu}_\infty^{\otimes k_n}) \longrightarrow 1.
\end{equation}
By letting $t$ and $s$ go to infinity in \cite[Theorem 2.3]{CLRW24}, we also have
$$\Ent(\mu^n_\infty |\bar{\mu}_\infty^{\otimes n}) \leq C,$$
so that
$$d_{TV}(\mu^{n,k_n}_\infty, \bar{\mu}_\infty^{\otimes k_n}) \leq C\frac{k_n}{n} \longrightarrow 0.$$
Therefore, 
\begin{align*}
	d_{TV}(\mu^n_{t_n}, \mu^n_\infty) &\geq d_{TV}(\mu^{n,k_n}_{t_n}, \mu^{n,k_n}_\infty)\\
&\geq d_{TV}(\bar{\mu}_{t_n}^{\otimes k_n}, \bar{\mu}_\infty^{\otimes k_n}) - d_{TV}(\bar{\mu}_{\infty}^{\otimes k_n}, \mu^{n,k_n}_\infty) - d_{TV}(\mu^{n,k_n}_{t_n}, \bar{\mu}_{t_n}^{\otimes k_n}) \longrightarrow 1.
\end{align*}
Hence 
$${\rm t}_{\textsc{mix}}^{(n)}  \geq t_n \gg \log \log n$$
which concludes the proof.\\
It remains to prove~\eqref{Eq:Claim}. Using the comparison between TV and Hellinger presented at the end of the proof of Theorem \ref{thm_cutoff_mean_field}, we deduce that
\begin{align*}
	d_{TV}(\bar{\mu}_{t_n}^{\otimes k_n}, \bar{\mu}_\infty^{\otimes k_n}) &\ge 1 - (1-\He(\bar{\mu}_t , \bar{\pi})^2)^{k_n}\\
	&\ge 1 - (1-d_{TV}(\bar{\mu}_{t_n}, \bar{\mu}_\infty)^2)^{k_n}\\
	&\ge 1 - e^{-k_n d_{TV}(\bar{\mu}_{t_n}, \bar{\mu}_\infty)^2}\;.
\end{align*}
Plugging the lower bound~\eqref{Eq:ToProveClaim} and using the explicit values $k_n = \sqrt{n}$ and $t_n = \sqrt{\log n}$, we deduce that this quantity goes to $1$.
\end{proof}

\begin{remark} \label{rem_non_quadratic_confinement}
In Corollary \ref{Cor:Cutoff_kinetic_mf_application}, we assumed that the confinement potential is quadratic only to get a lower bound on the mixing time. This assumption can be relaxed, at the price of having to derive information on an eigenfunction of the linearized evolution equation, as we shall now explain. 

As above let  $\bar{\mu}_\infty$ be the stationary measure of the non linear evolution and $\nu_t=\bar{\mu}_t-\bar{\mu}\infty$. Denote $L_\mu$ the generator of the kinetic McKean-Vlasov SDE where the interaction is with fixed $\mu$. Since $\mu\mapsto L_\mu$ is affine, every smooth $f$ satisfies the exact identity
$$\frac{\mathrm{d}}{\mathrm{d}t}\int f\,\mathrm{d}\nu_t=\int\mathcal{L}f\,\mathrm{d}\nu_t-Q_f(\nu_t).$$
Here
$$\mathcal{L}f(y,w)=L_{\bar{\mu}_\infty}f(y,w)-\int\nabla W(x-y)\cdot\nabla_v f(x,v)\,\bar{\mu}_\infty(\mathrm{d}x,\mathrm{d}v)$$
is the linearized generator, and
$$Q_f(\nu)=\iint\nabla W(x-y)\cdot\nabla_v f(x,v)\,\nu(\mathrm{d}x,\mathrm{d}v)\,\nu(\mathrm{d}y,\mathrm{d}w)$$
is quadratic, with $|Q_f(\nu)|\leq M_f\,W_1(\mu,\bar{\mu}_\infty)^2$ by a coupling argument.
Let $\psi$ be an eigenfunction of $\mathcal{L}$, i.e.\ $\mathcal{L}\psi=-c\,\psi$ up to an additive constant, with $c>0$ (whose existence is expected when $\hat{U}$ is sufficiently confining), and assume it is Lipschitz. Then
$$u(t):=\mathbb{E}\,\psi(X_t,V_t)-\int\psi\,\mathrm{d}\bar{\mu}_\infty$$
satisfies $u'=-cu-Q_\psi(\nu_t)$.
If $W_1(\bar{\mu}_t,\bar{\mu}_\infty)\leq C_0e^{-\rho t}\varepsilon$, with $\varepsilon=W_1(\bar{\mu}_0,\bar{\mu}_\infty)$ and $\rho>c/2$ (which is implied by $W_2$ contraction), Duhamel's formula gives
$u(t)\geq e^{-ct}\big(u(0)-C\varepsilon^2\big)$.
Hence $u(t)\geq\frac{\eta\varepsilon}{2}e^{-ct}$ whenever $u(0)\geq\eta\varepsilon$ and $\varepsilon$ is small enough.

This argument includes the situation of Corollary \ref{Cor:Cutoff_kinetic_mf_application}: if $\hat{U}$ is quadratic, then $\psi$ can be taken as linear in $(x,v)$; since $\nabla W$ is odd, $Q_\psi\equiv0$ and the bound is exact, with no smallness required.
\end{remark}

\bibliographystyle{alpha}
\bibliography{biblio}

\end{document}